\documentclass[11pt,oneside]{article}

\usepackage[utf8]{inputenc} \usepackage[T1]{fontenc}    \usepackage{url}            \usepackage{booktabs}       \usepackage{amsfonts}       \usepackage{nicefrac}       \usepackage{microtype}      

\usepackage{lmodern}

\usepackage{amssymb,amsmath,amsthm}
\usepackage{stmaryrd}
\usepackage[short]{optidef}

\usepackage{xcolor}
\usepackage{subfigure,graphicx,epsfig,tikz,caption}
\usepackage[normalem]{ulem}
\usepackage{enumerate}
\usepackage{verbatim}
\usepackage{makeidx,latexsym}
\usepackage[colorlinks=true]{hyperref}

\usepackage{bm}

\usepackage[numbers,sort&compress,square,comma]{natbib}

\usepackage{algorithmic,algorithm}

\usepackage{authblk}
\usepackage[in]{fullpage}
\makeatletter
\def\imod#1{\allowbreak\mkern10mu({\operator@font mod}\,\,#1)}
\makeatother
\usepackage[compact]{titlesec}

\usepackage[english]{babel}
\usepackage{bbm}
\usepackage{pgfplots}

\usepackage{parskip}
    \usepackage{geometry}
\usepackage{thmtools}

    \declaretheorem{theorem}
    \declaretheorem{corollary}
    \declaretheorem{lemma}
    \declaretheorem{proposition}
    \declaretheorem{observation}

    \declaretheoremstyle[qed=$\square$]{definitionwithend}
    \declaretheorem[style=definitionwithend]{definition}
    \declaretheorem[style=definitionwithend]{assumption}
    \declaretheorem[style=definitionwithend]{example}
    \declaretheorem[style=definitionwithend]{remark} \PassOptionsToPackage{numbers, compress}{natbib}

\definecolor{gold}{rgb}{0.85,0.65,0}

\newcommand{\be}{\begin{eqnarray}}
\newcommand{\ee}[1]{\label{#1}\end{eqnarray}}
\newcommand{\ese}{\end{eqnarray*}}
\newcommand{\bse}{\begin{eqnarray*}}
\def\beq{\begin{equation}}
\def\eeq{\end{equation}}

\def\fnote#1{\footnote}

\newcommand{\abs}[1]{\ensuremath{\left\lvert #1 \right\rvert}}

\newcommand{\by}{\times}
 
\newcommand{\norm}[1]{\ensuremath{\left\lVert #1 \right\rVert}}
\newcommand{\ip}[1]{\ensuremath{\left\langle #1 \right\rangle}}

\def\C{{\mathbb{C}}}

\def\N{{\mathbb{N}}}

\def\R{{\mathbb{R}}}
\def\S{{\mathbb{S}}}

\def\Se{{\mathbb{S}}}

\def\bS{{\mathbf{S}}}

\def\cA{{\cal A}}

\def\cD{{\cal D}}

\def\cF{{\cal F}}
\def\cG{{\cal G}}

\def\cV{{\cal V}}

\DeclareMathOperator{\Opt}{Opt}

\DeclareMathOperator*{\argmin}{arg\,min}
\DeclareMathOperator*{\argmax}{arg\,max}

\DeclareMathOperator{\Proj}{Proj}

\DeclareMathOperator{\rank}{rank}

\DeclareMathOperator{\Diag}{Diag}

\DeclareMathOperator{\tr}{tr}

\def\dim{\mathop{{\rm dim}\,}}

\DeclareMathOperator{\aff}{aff}

\DeclareMathOperator{\spann}{span}

\DeclareMathOperator{\inter}{int}

\DeclareMathOperator{\conv}{conv}

\DeclareMathOperator{\cone}{cone}

\def\extr{{\mathop{\rm extr}}}

    \let\emptyset\varnothing
    \newcommand{\set}[1]{\left\{#1\right\}}

\newcommand{\mb}{\mathbf}
    
    \newcommand{\mc}{\mathcal}
    \newcommand{\bb}{\mathbb}
     
\newcommand{\intset}[1]{\left\llbracket #1 \right\rrbracket}

\usepackage{tikz}

\renewcommand{\cA}{\bb{A}}
\newcommand{\mathprog}[1]{}

\begin{document}
\title{On the tightness of SDP relaxations of QCQPs
\thanks{This paper is an extended version of work published in IPCO 2020~\cite{wang2020convex}. Sections~\ref{subsec:socp},~\ref{subsec:sharpness},~\ref{sec:tightness}, and~\ref{sec:removing_polyhedral}, along with all of the proofs in this paper, are new material not present in the shorter version.}
}
\author[1]{Alex L.\ Wang}
\author[1]{Fatma K{\i}l{\i}n\c{c}-Karzan}
\affil[1]{Carnegie Mellon University, Pittsburgh, PA, 15213, USA.}
\date{\today}

\maketitle

\begin{abstract}
Quadratically constrained quadratic programs (QCQPs) are a fundamental class of optimization problems well-known to be NP-hard in general. In this paper we study conditions under which the standard semidefinite program (SDP) relaxation of a QCQP is tight. We begin by outlining a general framework for proving such sufficient conditions. Then using this framework, we show that the SDP relaxation is tight whenever the quadratic eigenvalue multiplicity, a parameter capturing the amount of symmetry present in a given problem, is large enough. We present similar sufficient conditions under which the projected epigraph of the SDP gives the convex hull of the epigraph in the original QCQP. Our results also imply new sufficient conditions for the tightness (as well as convex hull exactness) of a second order cone program relaxation of simultaneously diagonalizable QCQPs. \end{abstract}

\section{Introduction}
In this paper we study \textit{quadratically constrained quadratic programs} (QCQPs) of the following form
\begin{align}
\label{eq:qcqp}
\Opt \coloneqq \inf_{x\in\R^N}\set{q_0(x) :\, \begin{array}
	{l}
q_i(x)\leq 0 ,\,\forall i\in\intset{m_I}\\
q_i(x) = 0 ,\,\forall i\in\intset{m_I+1,m_I+m_E}
\end{array}
},
\end{align}
where for every $i\in\intset{0,m_I+m_E}$, the function $q_i:\R^N\to\R$ is a (possibly nonconvex) quadratic function.
We will write $q_i(x) = x^\top A_i x + 2b_i^\top x + c_i$ where $A_i\in \bb S^N$, $b_i\in\R^N$, and $c_i\in\R$.
Here $m_I$ and $m_E$ are the number of inequality constraints and equality constraints respectively. We will assume that $m\coloneqq m_I + m_E\geq 1$.

QCQPs arise naturally in many areas. A non-exhaustive list of applications contains facility location, production planning, pooling, max-cut, max-clique, and certain robust optimization problems (see \cite{phan1982quadratically,bao2011semidefinite,benTal2009robust} and references therein). More generally, any $\set{0,1}$~integer program or polynomial optimization problem may be reformulated as a QCQP~\cite{vandenberghe1996semidefinite}.

Although QCQPs are NP-hard to solve in general, they admit tractable convex relaxations. One natural relaxation is the standard (Shor) semidefinite program (SDP) relaxation \cite{shor1990dual}.
There is a vast literature on approximation guarantees associated with this relaxation \cite{ye1999approximating,nesterov1997quality,benTal2001lectures,megretski2001relaxations}, however, less is known about its exactness.
Recently, a number of exciting results in phase retrieval~\cite{candes2015phase} and clustering~\cite{mixon2016clustering,abbe2015exact,rujeerapaiboon2019size} 
have shown that under various assumptions on the data (or on the parameters in a random data model), the QCQP formulation of the corresponding problem has a tight SDP relaxation. See also \cite{luo2010semidefinite} and references therein for more examples of exactness results regarding SDP relaxations.
In contrast to these results, which address QCQPs arising from particular problems, \citet{burer2019exact} very recently gave some appealing deterministic sufficient conditions under which the standard SDP relaxation of \emph{general} QCQPs is tight. 
In our paper, 
we continue this vein of research for general QCQPs initiated by \citet{burer2019exact}.   
More precisely, we will provide sufficient conditions under which the following two types of results hold:  1) The convex hull of the epigraph of the QCQP is given by the
projection of the epigraph of its SDP relaxation, 2) the optimal objective value of the QCQP is equal to the optimal objective value of its SDP relaxation.
We will refer to these two types of results as ``convex hull results'' and ``SDP tightness results.''

The convex hull results will necessarily require stronger assumptions than the SDP tightness results, however they are also more broadly applicable because such convex hull results are  typically used as building blocks to derive strong convex relaxations for complex problems. 
In fact, the convexification of commonly occurring substructures has been critical in advancing the state-of-the-art computational approaches and software packages for mixed integer linear programs and general nonlinear nonconvex programs~\cite{conforti2014integer,tawarmalani2002convexification}.
For computational purposes, conditions guaranteeing simple convex hull descriptions are particularly favorable.
As we will discuss later, a number of our sufficient conditions will guarantee not only  the desired convex hull results but also that these convex hulls are given by a finite number of easily computable convex quadratic constraints in the original space of variables.

\subsection{Related work}

\subsubsection{Convex hull results}
Convex hull results are well-known for simple QCQPs such as the Trust Region Subproblem (TRS) and the Generalized Trust Region Subproblem (GTRS).
Recall that the TRS is a QCQP with a single strictly convex inequality constraint and that the GTRS is a QCQP with a single (possibly nonconvex) inequality constraint.
A celebrated result due to \citet{fradkov1979s-procedure} implies that the SDP relaxation of the GTRS is tight.
More recently, \citet{hoNguyen2017second} showed that the convex hull of the TRS epigraph is given exactly by the projection of the SDP epigraph. Follow-up work by \citet{wang2019generalized} showed that the (closed) convex hull of the GTRS epigraph is also given exactly by the projection of the SDP epigraph. 
In both cases, the projections of the SDP epigraphs can be described in the original space of variables with at most two convex quadratic inequalities. As a result, the TRS and the GTRS can be solved without explicitly running costly SDP-based algorithms; see \cite{adachi2019eigenvalue,jiang2020linear,jiang2019novel} for other algorithmic ideas to solve the TRS and GTRS.

A different line of research has focused on providing explicit descriptions for the convex hull of the intersection of a single nonconvex quadratic region with convex sets (such as convex quadratic regions, second-order cones~(SOCs), or polytopes) or with another single nonconvex quadratic region.
For example, the convex hull of the intersection of a two-term disjunction, which is a nonconvex quadratic constraint under mild assumptions, and the second-order cone (SOC) or its cross sections has received much attention in mixed integer programming  (see \cite{burer2017how,kilincKarzan2015two,yildiz2015disjunctive} and references therein).
\citet{burer2017how} also studied the convex hull of the intersection of a general nonconvex quadratic region with the SOC or its cross sections.
\citet{yildiran2009convex} gave an explicit description of the convex hull of the intersection of two \emph{strict} quadratic inequalities (note that the resulting set is open) under the mild regularity condition that there exists $\mu\in[0,1]$ such that $(1-\mu)A_0+\mu A_1\succeq 0$. Follow-up work by \citet{modaresi2017convex} gave sufficient conditions guaranteeing a closed version of the same result.  
More recently, \citet{santana2018convex} gave an explicit description of the convex hull of the intersection of a nonconvex quadratic region with a polytope; this convex hull was further shown to be second-order cone representable.
In contrast to these results, we will not limit the number of nonconvex quadratic constraints in our QCQPs. Additionally, the nonconvex sets that we study in this paper will arise as epigraphs of QCQPs. In particular, the epigraph variable will play a special role in our analysis. Therefore, we view our developments as complementary to these results.

The convex hull question has also received attention for certain strengthened relaxations of simple QCQPs \cite{sturm2003cones,burer2013second,burer2014trust,burer2015gentle}.
In this line of work, the standard SDP relaxation is strengthened by additional inequalities derived using the Reformulation-Linearization Technique (RLT). 
\citet{sturm2003cones} showed that the standard SDP relaxation strengthened with an additional SOC constraint derived from RLT gives the convex hull of the epigraph of the TRS with one additional linear inequality. \citet{burer2014trust} extended this result to the case of an arbitrary number of additional linear inequalities as long as the linear constraints do not intersect inside the trust region domain. See \cite{burer2015gentle} for a survey of some results in this area.
Note that in this paper, we restrict our attention to the standard SDP relaxation of QCQPs. Nevertheless, establishing exactness conditions for strengthened SDP relaxations of QCQPs is clearly of great interest and is a direction for future research.

\subsubsection{SDP tightness results}

A number of SDP tightness results are known for variants of the TRS.
\citet{jeyakumar2013trust} showed that the standard SDP relaxation of the TRS with additional linear inequalities is tight under a condition regarding the dimension of 
the minimum eigenvalue\footnote{More precisely, this is the minimum generalized eigenvalue of $A_0$ with respect to the positive definite quadratic form in the constraint.} of $A_0$.
These results were extended in the same paper to handle multiple convex quadratic inequality constraints with the same sufficiently rank-deficient quadratic form (see \cite[Section 6]{jeyakumar2013trust}). \citet{hoNguyen2017second} presented a sufficient condition for tightness of the SDP relaxation that is slightly more general than \cite[Section 6]{jeyakumar2013trust} (see \citet[Section 2.2]{hoNguyen2017second} for a comparison of these conditions). 
A related line of work by \citet{ye2003new} and \citet{beck2006strong} gives sufficient conditions under which the TRS with one additional quadratic inequality constraint admits a tight SDP relaxation.
In contrast to this line of work, our results will address the SDP tightness question in the context of more general QCQPs.

In terms of SDP tightness results, simultaneously diagonalizable QCQPs (SD-QCQPs) have received separate attention \cite{locatelli2015some,locatelli2016exactness,benTal2014hidden,jiang2016simultaneous}.
It is shown in \cite[Theorem 2.1]{locatelli2016exactness} that for SD-QCQPs, the SDP relaxation is equivalent to a SOC program (SOCP) relaxation (see also Proposition~\ref{prop:SOCP_SDP_equivalence}). 
In particular, the KKT-based sufficient conditions that have been presented for SOCP tightness in \cite{benTal2014hidden,locatelli2015some} also guarantee SDP tightness. 
We will present SDP tightness results (Theorems~\ref{thm:sdp_tightness_main}~and~\ref{thm:sdp_tightness_perturbed}) that generalize some of the conditions presented in this line of work. More specifically, our results will not make use of simultaneous diagonalizability assumptions.

A series of articles beginning with \citet{beck2007quadratic} and \citet{beck2012new} has derived SDP tightness results for quadratic matrix programs (QMPs). A QMP is an optimization problem of the form
\begin{align*}
\inf_{X\in\R^{n\by k}}\set{\tr(X^\top \cA_0 X)+2\tr(B_0^\top X) + c_0: \, \begin{array}
	{r}
	\tr(X^\top \cA_i X)+2\tr(B_i^\top X) + c_i \leq 0,\quad\\
	\forall i\in\intset{m_I}\\
	\tr(X^\top \cA_i X)+2\tr(B_i^\top X) + c_i = 0,\quad\\
	\forall i\in\intset{m_I+1,m}
\end{array}},
\end{align*}
where $\cA_i\in \bb S^n$, $B_i\in\R^{n\by k}$, and $c_i \in\R$, and arises often in robust least squares or as a result of Burer-Monteiro reformulations for rank-constrained semidefinite programming \cite{beck2007quadratic,burer2003nonlinear}. 
In this research vein, \citet{beck2007quadratic} showed that a carefully constructed SDP relaxation of QMP is tight whenever $m \leq k$.
Note that by replacing the matrix variable $X\in\R^{n\by k}$ by the vector variable $x\in\R^{nk}$, we may reformulate any QMP as a QCQP of a very particular form. Working backwards, if a QCQP can be reformulated as a QMP with $m\leq k$, then we may apply the SDP relaxation proposed in \cite{beck2007quadratic} to solve it exactly.
We will discuss how such a condition compares with our assumptions in Section~\ref{sec:symmetries}.

In a recent intriguing paper, \citet{burer2019exact} gave a sufficient condition guaranteeing that the standard SDP relaxation of general QCQPs is tight.
We emphasize that in contrast to prior work, the condition proposed in~\cite{burer2019exact} can be applied to \textit{general} QCQPs.
Then, motivated by recent results on exactness guarantees for specific recovery problems with random data and sampling, 
\citet{burer2019exact} also examined a class of random QCQPs and established that if the number of constraints $m$ grows no faster than any fixed polynomial in the number of variables $N$, then their sufficient condition holds with probability approaching one. In particular, the SDP relaxation is tight with probability approaching one.
The SDP tightness results that we present (Theorems~\ref{thm:sdp_tightness_main}~and~\ref{thm:sdp_tightness_perturbed}) will generalize their deterministic sufficient condition \cite[Theorem 1]{burer2019exact}.
As such, their proofs directly imply that our sufficient conditions also hold with probability approaching one in their random data model.

\subsection{Overview and outline of paper}

In contrast to the literature, which has mainly focused on simple QCQPs or QCQPs under certain structural assumptions, in this paper, we will consider general QCQPs and develop sufficient conditions for both the convex hull result and the SDP tightness result. 

We first introduce the epigraph of the QCQP by writing
\begin{align*}
\Opt &= \inf_{(x,t)\in\R^{N+1}} \set{2t :\, (x,t)\in\mc D},
\end{align*}
where $\mc D$ is the epigraph of the QCQP in \eqref{eq:qcqp}, i.e.,
\begin{align}
\label{eq:qcqp_epi}
\mc D \coloneqq \set{(x,t)\in\R^N\times \R :\, \begin{array}
	{l}
	q_0(x) \leq 2t\\
	q_i(x) \leq 0 ,\,\forall i\in\intset{m_I}\\
	q_i(x) = 0,\,\forall i\in\intset{m_I+1,m}
\end{array}}.
\end{align}
As $(x,t)\mapsto 2t$ is linear, we may replace the (potentially nonconvex) epigraph $\mc D$ with its convex hull $\conv(\mc D)$. Then,
\begin{align*}
\Opt &= \inf_{(x,t)\in\R^{N+1}} \set{2t :\, (x,t)\in\conv(\mc D)}.
\end{align*}

A summary of our contributions, along with an outline of the paper, is as follows:
\begin{enumerate}[(a)]
\item In Section~\ref{sec:framework}, we introduce and study the standard SDP relaxation of QCQPs \cite{shor1990dual} along with its optimal value $\Opt_\textup{SDP}$ and projected epigraph $\mc D_\textup{SDP}$. We set up a framework for deriving sufficient conditions for the ``convex hull result,'' $\conv(\mc D) = \mc D_\textup{SDP}$, and the ``SDP tightness result,'' $\Opt=\Opt_\textup{SDP}$. This framework is based on the Lagrangian function $(\gamma,x)\mapsto q_0(x) + \sum_{i=1}^m \gamma_i q_i(x)$ and the eigenvalue structure of a dual object $\Gamma\subseteq\R^m$.
This object $\Gamma$, which consists of the convex Lagrange multipliers, has been extensively studied in the literature (see \cite[Chapter 13.4]{wolkowicz2012handbook} and more recently \cite{sheriff2013convexity}).

\item In Section~\ref{sec:symmetries}, we define an integer parameter $k$, the quadratic eigenvalue multiplicity, that captures the amount of symmetry in a given QCQP. We then give examples where the quadratic eigenvalue multiplicity is large. Specifically, vectorized reformulations of quadratic matrix programs \cite{beck2007quadratic} are such an example.

\item In Section~\ref{sec:conv_hull}, we use our framework to derive sufficient conditions for the convex hull result:~$\conv(\mc D) = \mc D_\textup{SDP}$. Theorem~\ref{thm:conv_hull_symmetries} states that if $\Gamma$ is polyhedral and $k$ is sufficiently large, then $\conv(\mc D) = \mc D_\textup{SDP}$. This theorem  actually follows as a consequence of Theorem~\ref{thm:conv_hull_main}, which replaces the assumption on the quadratic eigenvalue multiplicity with a weaker assumption regarding the dimension of zero eigenspaces related to the matrices $A_i$.
Furthermore, our results in this section establish that if $\Gamma$ is polyhedral, then $\mc D_\textup{SDP}$ is SOC representable; see Remark~\ref{rem:soc_representability}. In particular, when the assumptions of Theorems~\ref{thm:conv_hull_main}~or~\ref{thm:conv_hull_symmetries} hold, we have that $\conv(\mc D) = \mc D_\textup{SDP}$ is SOC representable.
In Section~\ref{subsec:applications_of_conv_hull}, we provide several classes of problems that satisfy the assumptions of these theorems.
In particular, we recover a number of results regarding the TRS~\cite{hoNguyen2017second}, the GTRS~\cite{wang2019generalized}, and the solvability of systems of quadratic equations \cite{barvinok1993feasibility}.
In Section~\ref{subsec:socp}, we compare our assumption that $\Gamma$ is polyhedral with the assumption that the QCQP is an SD-QCQP and show that our assumption is strictly more general.
In Section~\ref{subsec:socp_rel_SD-QCQPs}, we prove that the SOCP relaxation of SD-QCQPs considered by \cite{benTal2014hidden} and \citet{locatelli2015some,locatelli2016exactness} is indeed equivalent to the SDP relaxation. Consequently, this allows us to recover some of the results from \cite{benTal2014hidden,locatelli2015some,locatelli2016exactness} as a consequence of our sufficient conditions (see Section~\ref{subsec:literature_comparison}).
In Section~\ref{subsec:sharpness}, we conclude by showing that the dependence we prove on the quadratic eigenvalue multiplicity $k$ is optimal (Propositions~\ref{prop:aff_dim_sharp_example}~and~\ref{prop:Gamma_sharp_example}).

\item In Section~\ref{sec:tightness}, we use our framework to derive sufficient conditions for the SDP tightness result:~$\Opt =\Opt_\textup{SDP}$. Specifically, Theorems~\ref{thm:sdp_tightness_main}~and~\ref{thm:sdp_tightness_perturbed} give generalizations of the conditions introduced by \citet{locatelli2016exactness} for SDP tightness in a variant of the TRS and \citet{burer2019exact} for SDP tightness in diagonal QCQPs.

\item In Section~\ref{sec:removing_polyhedral}, we discuss the assumption that the dual object $\Gamma$ is polyhedral. In particular, we show that it is possible to recover both a convex hull result (Theorem~\ref{thm:general_Gamma_conv_hull}) and an SDP tightness result (Theorem~\ref{thm:general_Gamma_sdp_tightness}) when this assumption is dropped as long as the quadratic eigenvalue multiplicity $k$ is sufficiently large.
\end{enumerate}

To the best of our knowledge, our results are the first to provide a unified explanation of many of the exactness guarantees present in the literature.
Moreover, our results also provide significant generalizations in a number of settings.
We discuss the relevant comparisons in detail in the corresponding sections as outlined above.
Finally, our results present the first sufficient conditions under which the convex hull of the epigraph of a general QCQP is SOC representable.
 \subsection{Notation}
Let $\R_+$ denote the nonnegative reals. 
For nonnegative integers $m\leq n$ let $\intset{n}\coloneqq\{1,\ldots,n\}$ and $\intset{m,n}\coloneqq \set{m,m+1,\dots,n-1,n}$. 
For $i\in\intset{n}$, let $e_i\in\R^n$ denote the $i$th standard basis vector.
Let $\mb{S}^{n-1}=\set{x\in\R^n:\, \norm{x}=1}$ denote the $n-1$ sphere
and let $\mb{B}(\bar{x},r)=\set{x\in\R^n:\, \norm{x-\bar{x}} \leq r }$ denote the $n$-ball with radius $r$ and center $\bar x$. 
Let $\Se^n$ denote the set of real symmetric $n\by n$ matrices.
For a positive integer $n$, let $I_n$ denote the $n\times n$ identity matrix. 
When the dimension is clear from context, we will simply write $I$ instead of $I_n$. 
Given $A\in \Se^n$, let $\det(A)$, $\tr(A)$, and $\lambda_{\min}(A)$ denote the determinant, trace, and minimum eigenvalue of $A$, respectively.
We write $A\succeq 0$ (respectively, $A\succ 0$) if $A$ is positive semidefinite (respectively, positive definite). 
For $a\in\R^n$, let $\Diag(a)$ denote the diagonal matrix $A\in\R^{n\by n}$ with diagonal entries $A_{i,i} = a_i$. 
Given two matrices $A$ and $B$, let $A \otimes B$ denote their Kronecker product.  
For a set $\mc D \subseteq \R^n$, let $\inter(\mc D)$, $\conv(\mc D)$, $\cone(\mc D)$, $\extr(\mc D)$, $\spann(\mc D)$, $\aff(\mc D)$, $\dim(\mc D)$ and $\aff\dim(\mc D)$ denote the interior, convex hull, conic hull, extreme points, span, affine span, dimension, and affine dimension of $\mc D$, respectively.
For a subspace of $V$ of $\R^n$ and $x\in\R^n$, let $\Proj_{V}x$ denote the projection of $x$ onto $V$. 

\section{A general framework}
\label{sec:framework}

In this section, we introduce a general framework for analyzing the standard Shor SDP relaxation of QCQPs. We will examine how both the objective value and the feasible domain change when moving from a QCQP to its SDP relaxation.

We make an assumption that can be thought of as a primal feasibility and dual strict feasibility assumption. This assumption (or a slightly stronger version of it) is standard and is routinely made in the literature on QCQPs (see for example \cite{benTal1996hidden,ye2003new,beck2007quadratic}).
\begin{assumption}
\label{as:gamma_definite}
	Assume the feasible region of \eqref{eq:qcqp} is nonempty and there exists $\gamma^*\in\R^m$ such that $\gamma^*_i\geq 0$ for all $i\in\intset{m_I}$ and $A_0 + \sum_{i=1}^m \gamma^*_i A_i \succ 0$.\mathprog{\qed}
\end{assumption}
\begin{remark}\label{rem:gamma_star_positive}
By the continuity of $\gamma\mapsto\lambda_{\min}(A_0 + \sum_{i=1}^m \gamma_i A_i)$, we may assume without loss of generality that $\gamma^*_i>0$ for all $i\in\intset{m_I}$.\mathprog{\qed}
\end{remark}

The standard SDP relaxation of~\eqref{eq:qcqp} takes the following form
\begin{align}
\label{eq:shor_sdp}
\Opt_\textup{SDP} \coloneqq \inf_{x\in\R^N, X\in \Se^{N}} \set{\ip{Q_0, Y}:\, \begin{array}
	{l}
	Y\coloneqq \begin{pmatrix}
		1 & x^\top \\ x & X
	\end{pmatrix}\\
	\ip{Q_i, Y}\leq 0 ,\,\forall i\in\intset{m_I}\\
	\ip{Q_i, Y}= 0 ,\,\forall i\in\intset{m_I+1,m}\\
	Y\succeq 0
\end{array}}.
\end{align}
Here, $Q_i\in\Se^{N+1}$ is the matrix $Q_i \coloneqq \left(\begin{smallmatrix}
	c_i & b_i^\top \\
	b_i & A_i
\end{smallmatrix}\right)$.
Let $\mc D_{\textup{SDP}}$ denote the epigraph of \eqref{eq:shor_sdp} projected onto the $(x,t)$ variables, i.e., define
\begin{align}
\label{eq:sdp_epi}
\mc D_{\textup{SDP}} \coloneqq \set{(x,t) \in\R^{N+1}:\, \begin{array}
	{l}
	\exists X\in \Se^N:\\
	Y \coloneqq\begin{pmatrix}
		1 & x^\top\\ x & X
	\end{pmatrix}\\
	\ip{Q_0, Y} \leq 2t\\
	\ip{Q_i, Y}\leq 0 ,\,\forall i\in\intset{m_I}\\
	\ip{Q_i, Y}= 0 ,\,\forall i\in\intset{m_I+1,m}\\
	Y\succeq 0
\end{array}}.
\end{align}

By taking $X = xx^\top$ in both \eqref{eq:shor_sdp} and \eqref{eq:sdp_epi}, we see that $\mc D \subseteq \mc D_\textup{SDP}$ and $\Opt\geq \Opt_\textup{SDP}$. Noting that $\mc D_\textup{SDP}$ is convex (it is the projection of a convex set), we further have that $\conv(\mc D)\subseteq \mc D_{\textup{SDP}}$. The framework that we set up in the remainder of this section allows us to reason about when equality occurs in both relations, i.e., when $\conv(\mc D) = \mc D_\textup{SDP}$ and/or $\Opt = \Opt_\textup{SDP}$. We will refer to these two types of result as ``convex hull results'' and ``SDP tightness results.''

\subsection{Rewriting the SDP in terms of a dual object}

For $\gamma\in\R^m$, define
\begin{align*}
A(\gamma)\coloneqq A_0 + \sum_{i=1}^m \gamma_i A_i,
\quad
b(\gamma)\coloneqq b_0 + \sum_{i=1}^m \gamma_i b_i,
\quad
c(\gamma)\coloneqq c_0 + \sum_{i=1}^m \gamma_i c_i,
\quad
q(\gamma,x)\coloneqq q_0(x) + \sum_{i=1}^m \gamma_i q_i(x).
\end{align*}
It is easy to verify that $q(\gamma,x) = x^\top A(\gamma)x + 2b(\gamma)^\top x + c(\gamma)$.
Our framework for analyzing \eqref{eq:shor_sdp} is based on the dual object
\begin{align*}
\Gamma\coloneqq \set{\gamma\in\R^m:\, \begin{array}
	{l}
	A(\gamma)\succeq 0\\
	\gamma_i\geq 0 ,\,\forall i\in\intset{m_I}
\end{array}}.
\end{align*}

We begin by rewriting both $\mc D_\textup{SDP}$ and $\Opt_\textup{SDP}$ to highlight the role played by $\Gamma$. 

\begin{lemma}
\label{lemma:sdp_in_terms_of_Gamma}
Suppose Assumption~\ref{as:gamma_definite} holds. Then
\begin{align*}
\mc D_\textup{SDP}  = \set{(x,t) :~ \sup_{\gamma\in\Gamma} q(\gamma,x)\leq 2t}
\quad\text{and}\quad
\Opt_\textup{SDP} &= \min_{x\in\R^N} \sup_{\gamma\in\Gamma} q(\gamma, x).
\end{align*}
\end{lemma}
We note that the second identity in Lemma~\ref{lemma:sdp_in_terms_of_Gamma} is well-known and was first recorded by \citet{fujie1997semidefinite}.
\begin{proof}
The second identity follows immediately from the first identity, thus it suffices to prove only the former.

Fix $\hat x$ and consider the SDP
\begin{align}
\label{eq:conv_hull_sdp}
\inf_{X\in\Se^N}\set{\ip{Q_0, Y} :\, \begin{array}
	{l}
	Y\coloneqq \begin{pmatrix}
		1 & \hat x^\top \\
		\hat x & X
	\end{pmatrix}\\
	\ip{Q_i,Y} \leq 0 ,\,\forall i\in\intset{m_I}\\
	\ip{Q_i,Y} = 0 ,\,\forall i\in\intset{m_I+1,m}\\
	Y\succeq 0
\end{array}}.
\end{align}
Comparing programs \eqref{eq:sdp_epi} and \eqref{eq:conv_hull_sdp}, we see that $(\hat x,\hat t)\in\mc D_\textup{SDP}$ if and only if the value $2\hat t$ is achieved in \eqref{eq:conv_hull_sdp}. The dual SDP to \eqref{eq:conv_hull_sdp} is given by
\begin{align}
\label{eq:conv_hull_sdp_dual}
\sup_{\gamma\in\R^m, t\in\R, y\in\R^N}\set{2t+2\ip{y,\hat x}:\, \begin{array}
	{l}
	\begin{pmatrix}
		c(\gamma) - 2t & b(\gamma)^\top - y^\top\\
		b(\gamma) - y & A(\gamma)
	\end{pmatrix}\succeq 0\\
	\gamma_i \geq 0 ,\,\forall i\in\intset{m_I}
\end{array}}.
\end{align}
Note that the first constraint in the dual SDP can only be satisfied if $A(\gamma)\succeq 0$. We may thus rewrite
\begin{align*}
\eqref{eq:conv_hull_sdp_dual} &= \sup_{\gamma\in\R^m, t\in\R, y\in\R^N} \set{2t+2\ip{y,\hat x} :\, \begin{array}
	{l}
	\begin{pmatrix}
		1\\x
	\end{pmatrix}^\top \begin{pmatrix}
		c(\gamma) - 2t & b(\gamma)^\top - y^\top \\
		b(\gamma)-y & A(\gamma)
	\end{pmatrix}\begin{pmatrix}
		1\\x
	\end{pmatrix}\geq 0,\,\forall x\in\R^N\\
	\gamma\in\Gamma
\end{array}}\\
&= \sup_{\gamma\in\R^m, t\in\R, y\in\R^N} \set{2t+2\ip{y,\hat x} :\, \begin{array}
	{l}
	q(\gamma,x)- 2\ip{y,x} \geq 2t  ,\,\forall x\in\R^N\\
	\gamma\in\Gamma
\end{array}}\\
&= \sup_{\gamma\in\Gamma, y\in\R^N} \inf_{x\in\R^N} q(\gamma,x) +2\ip{y,\hat x - x}.
\end{align*}

We first consider the case that the value of the dual SDP~\eqref{eq:conv_hull_sdp_dual} is bounded. Assumption~\ref{as:gamma_definite} and Remark~\ref{rem:gamma_star_positive} imply that \eqref{eq:conv_hull_sdp_dual} is strictly feasible. Then by strong conic duality, the primal SDP~\eqref{eq:conv_hull_sdp} achieves its optimal value and in particular must be feasible.
Let $\gamma^*$ be such that $A(\gamma^*)\succ 0$ (this exists by Assumption~\ref{as:gamma_definite}) and let $y^* = 0$. Then,
\begin{align*}
\lim_{\norm{x}\to\infty} q(\gamma^*,x) + 2\ip{y^*, \hat x - x} = \lim_{\norm{x}\to\infty} q(\gamma^*,x) = \infty.
\end{align*}
In other words, $x\mapsto q(\gamma^*,x) + 2 \ip{y^*, \hat x - x}$ is coercive and we may apply the Minimax Theorem \cite[Chapter VI, Proposition 2.3]{ekeland1999convex} to get
\begin{align*}
\eqref{eq:conv_hull_sdp}= \eqref{eq:conv_hull_sdp_dual} &= \min_{x\in\R^N}\sup_{\gamma\in\Gamma, y\in\R^N}  q(\gamma,x) +2\ip{y,\hat x - x} = \sup_{\gamma\in\Gamma} q(\gamma,\hat x).
\end{align*}
The last equation follows as for any $x\neq\hat x$, the supremum may take $y$ arbitrarily large in the direction of $\hat x - x$. We conclude that if the value of the dual SDP~\eqref{eq:conv_hull_sdp_dual} is bounded, then
\begin{align*}
(\hat x,\hat t)\in \mc D_\textup{SDP} \quad\Longleftrightarrow\quad \sup_{\gamma\in\Gamma} q(\gamma,\hat x) \leq 2\hat t.
\end{align*}

Now suppose the value of the dual SDP~\eqref{eq:conv_hull_sdp_dual} is unbounded. In this case $(\hat x,\hat t)\notin \mc D_\textup{SDP}$ for any value of $\hat t$. It remains to observe that
\begin{align*}
\sup_{\gamma\in\Gamma} q(\gamma,\hat x) \geq \sup_{\gamma\in\Gamma, y\in\R^N} \inf_{x\in\R^N} q(\gamma, x) + 2\ip{y, \hat x - x}=\infty.
\end{align*}
In particular, $(\hat x,\hat t)$ does not satisfy $\sup_{\gamma\in\Gamma} q(\gamma,\hat x)\leq 2\hat t$ for any value of $\hat t$. We conclude that if the value of the dual SDP \eqref{eq:conv_hull_sdp_dual} is unbounded, then for all $\hat t$,
\begin{align*}
(\hat x, \hat t)\notin\mc D_\textup{SDP} &\quad\text{and}\quad \sup_{\gamma\in\Gamma}q(\gamma,\hat x) \not\leq 2\hat t.\qedhere
\end{align*}
\end{proof}

\begin{remark}
\label{rem:affine_transformation}
It is not hard to show\footnote{A short proof follows from Lemma~\ref{lemma:sdp_in_terms_of_Gamma}.} that the questions $\conv(\mc D)\stackrel{?}{=} \mc D_\textup{SDP}$ and $\Opt\stackrel{?}{=} \Opt_\textup{SDP}$ are invariant under invertible affine transformations of the $x$-space.
In particular, the sufficient conditions that we will present in this paper only need to hold after some invertible affine transformation. In this sense, the SDP relaxation will ``find'' structure in a given QCQP even if it is ``hidden'' by an affine transformation.\mathprog{\qed}
\end{remark}

\subsection{The eigenvalue structure of $\Gamma$}
We will make a technical assumption on $\Gamma$ and $q(\gamma,x)$ in the remainder of our framework.

\begin{assumption}
\label{as:cF_well_defined}
	Assume that for all $\hat x\in\R^n$, if $\sup_{\gamma\in\Gamma}q(\gamma,\hat x)$ is finite then its maximum value is achieved in $\Gamma$.\mathprog{\qed}
\end{assumption}
\begin{remark}\label{rem:F(x)_well_defined}
As $\gamma\mapsto q(\gamma,\hat x)$ is linear in $\gamma$ and $\Gamma$ is closed, Assumption~\ref{as:cF_well_defined} holds for example whenever $\Gamma$ is polyhedral or bounded.\mathprog{\qed}
\end{remark}

Under Assumption~\ref{as:cF_well_defined}, the following definition is well-defined.
\begin{definition}
Suppose Assumption~\ref{as:cF_well_defined} holds.
For any $\hat x\in\R^N$ such that $\sup_{\gamma\in\Gamma} q(\gamma,\hat x)$ is finite, define $\mc F(\hat x)$ to be the face of $\Gamma$ achieving $\sup_{\gamma\in\Gamma}q(\gamma,\hat x)$, i.e.,
\begin{align*}
\cF(\hat x) &\coloneqq \argmax_{\gamma\in\Gamma}q(\gamma,x).\qedhere
\end{align*}
\end{definition}
\begin{definition}\label{def:definiteFace}
Let $\mc F$ be a face of $\Gamma$. We say that $\mc F$ is a \textit{definite face} if there exists $\gamma\in \mc F$ such that $A(\gamma)\succ 0$. Otherwise, we say that $\mc F$ is a \textit{semidefinite face} and let $\mc V(\mc F)$ denote the shared zero eigenspace of $\mc F$, i.e.,
\begin{align*}
\mc V(\mc F) &\coloneqq \set{v\in\R^N:\, A(\gamma)v = 0,\,\forall \gamma\in\mc F}.\qedhere
\end{align*}
\end{definition}
Note that under Definition~\ref{def:definiteFace}, each face of $\Gamma$ is \textit{either} a definite face or a semidefinite face. Specifically, a definite face is not also a semidefinite face.

The following lemma shows that $\cV(\cF)$, which \textit{a priori} may be the trivial subspace $\set{0}$, in fact contains nonzero elements when $\cF$ is a semidefinite face.

\begin{lemma}
\label{lem:semidefinite_shares_zero}
Let $\mc F$ be a semidefinite face of $\Gamma$. Then $\mc V(\mc F)\cap \mb S^{N-1}$ is nonempty.
\end{lemma}
\begin{proof}
Let $\hat\gamma$ denote a vector in the relative interior of $\cF$. By the assumption that $\cF$ is a semidefinite face, there exists $v\in\bS^{N-1}$ such that $v^\top A(\hat\gamma)v = 0$. We claim that $v\in\cV(\cF)$. 
As $A(\gamma)\succeq 0$ for all $\gamma\in\cF$, it suffices to show that $v^\top A(\gamma)v\leq 0$ for all $\gamma\in\cF$.
Suppose $v^\top A(\gamma')v>0$ for some $\gamma'\in\cF$. Then, as $\hat\gamma$ is in the relative interior of $\cF$, there exists $\epsilon>0$ small enough such that $\gamma_\epsilon\coloneqq \hat\gamma + \epsilon(\hat\gamma - \gamma')\in\cF$. Finally, by the linearity of $\gamma\mapsto v^\top A(\gamma)v$ in $\gamma$, we conclude that $v^\top A(\gamma_\epsilon)v <0$, a contradiction.
\end{proof}

\subsection{The framework}
Our framework for analyzing the SDP relaxation consists of an ``easy part'' and a ``hard part.''
The former only requires Assumptions~\ref{as:gamma_definite}~and~\ref{as:cF_well_defined} to hold while the latter may require additional assumptions.
We detail the ``easy part'' in the remainder of this section.

Begin by making the following observations.
\begin{lemma}
\label{lem:F_contains_definite}
Suppose Assumptions~\ref{as:gamma_definite}~and~\ref{as:cF_well_defined} hold and let $(\hat x, \hat t) \in\mc D_\textup{SDP}$. If $\mc F(\hat x)$ is a definite face of $\Gamma$, then $(\hat x, \hat t)\in\mc D$.
\end{lemma}
\begin{proof}
Let $\mc F\coloneqq \mc F(\hat x)$. Because $\mc F$ is a definite face, there exists $\gamma^*\in\mc F$ such that $A(\gamma^*)\succ 0$. We verify that $(\hat x,\hat t)$ satisfies each of the constraints in \eqref{eq:qcqp_epi}.

\begin{enumerate}
	\item By continuity, there exists $\epsilon>0$ such that
$A((1+\epsilon)\gamma^*)\succ 0$.
We claim that $(1+\epsilon)\gamma^*\in\mc F$.
Indeed, $A(\gamma^*)$ and $A((1+\epsilon)\gamma^*)$ are both positive definite, thus the constraint $A(\gamma)\succeq 0$ is inactive at both $\gamma^*$ and $(1+\epsilon)\gamma^*$. Furthermore, for all $i\in\intset{m_I}$, the constraint $\gamma_i\geq 0$ is active at $\gamma^*$ if and only if it is active at $(1+\epsilon)\gamma^*$.
We conclude that $(1+\epsilon)\gamma^*\in\mc F$ and in particular $0\in\aff(\mc F)$. This implies
\begin{align*}
q_0(\hat x) &= q(0,\hat x) = q(\gamma^*,\hat x)\leq 2\hat t.
\end{align*}
	\item Let $i\in\intset{m_I}$. By continuity
	there exists $\epsilon>0$ such that $A(\gamma^* + \epsilon e_i)\succ 0$. Thus, $\gamma^* + \epsilon e_i\in\Gamma$. In particular, since $q(\gamma,\hat x)$ is maximized on $\mc F$ in $\Gamma$, we have that
\begin{align*}
q_i(\hat x) =\frac{q(\gamma^* + \epsilon e_i,\hat x) - q(\gamma^*, \hat x)}{\epsilon} \leq 0.
\end{align*}
\item Let $i\in\intset{m_I+1,m}$. By continuity, there exists $\epsilon>0$ such that
$A(\gamma^* \pm \epsilon e_i)\succ 0$.
Thus, $\gamma^* \pm \epsilon e_i\in\Gamma$. In particular, since $q(\gamma,\hat x)$ is maximized on $\mc F$ in $\Gamma$, we have that
\begin{align*}
q_i(\hat x) =\frac{q(\gamma^* + \epsilon e_i,\hat x) - q(\gamma^*, \hat x)}{\epsilon} \leq 0.
\end{align*}
Repeating this calculation with $-\epsilon$ gives $q_i(\hat x)\geq 0$. We deduce that $q_i(\hat x) = 0$.\qedhere
\end{enumerate}
\end{proof}

\begin{observation}
\label{obs:aff_dim_m_is_definite}
Suppose Assumption~\ref{as:gamma_definite} holds, and let $\mc F$ be a face of $\Gamma$. If $\aff\dim(\mc F) = m$, then $\mc F$ is definite.
\end{observation}

Together, Lemma~\ref{lem:F_contains_definite} and Observation~\ref{obs:aff_dim_m_is_definite} give a sufficient condition for a point $(\hat x,\hat t)\in\cD_\textup{SDP}$ to belong to $\cD$, namely when $\aff\dim(\cF(\hat x)) = m$. Concretely, we can use the quantity $\aff\dim(\cF(\hat x))$ to measure the progress of a convex decomposition algorithm.

\begin{lemma}
\label{lem:applied_easy_part_conv_hull}
Suppose Assumptions~\ref{as:gamma_definite}~and~\ref{as:cF_well_defined} hold. Suppose furthermore that:
\begin{equation}
  \label{eq:applied_easy_part_conv_hull}
  \parbox{\dimexpr\linewidth-4em}{\strut
    For every $(\hat x,\hat t)\in\cD_\textup{SDP}$ with $\cF(\hat x)$ semidefinite, we can write $(\hat x,\hat t)$ as a convex combination of points $(x_\alpha,t_\alpha)\in\cD_\textup{SDP}$ such that $\aff\dim(\cF(x_\alpha))>\aff\dim(\cF(\hat x))$.
    \strut}
\end{equation}
Then $\conv(\cD) = \cD_\textup{SDP}$ and $\Opt=\Opt_\textup{SDP}$.
\end{lemma}

\begin{proof}
Suppose for the sake of contradiction that $\conv(\cD)\neq \cD_\textup{SDP}$. Let
\begin{align*}
(\hat x,\hat t) \in\argmax_{\cD_\textup{SDP}\setminus\conv(\cD)} \left(\aff\dim(\cF(\hat x))\right).
\end{align*}
The point $(\hat x,\hat t)$ is well-defined as $\aff\dim(\cF(\hat x))$ is a nonnegative integer bounded above by $m-1$ (this follows from Lemma~\ref{lem:F_contains_definite} and Observation~\ref{obs:aff_dim_m_is_definite}).
By Lemma~\ref{lem:F_contains_definite}, we must have that $\cF(\hat x)$ is semidefinite. By \eqref{eq:applied_easy_part_conv_hull}, there exist points $(x_\alpha,t_\alpha)\in\cD_\textup{SDP}$ such that $\aff\dim(\cF(x_\alpha))>\aff\dim(\cF(\hat x))$. Then by construction of $(\hat x,\hat t)$ and the fact that $\aff\dim(\cF(x_\alpha))>\aff\dim(\cF(\hat x))$, we have that $(x_\alpha,t_\alpha)\in\conv(\cD)$. We conclude that
$(\hat x,\hat t)\in \conv(\set{(x_\alpha,t_\alpha)}_\alpha) \subseteq \conv(\cD)$, a contradiction.
\end{proof}
Equivalently, when the assumptions of Lemma~\ref{lem:applied_easy_part_conv_hull} hold, the following convex decomposition procedure is guaranteed to terminate and succeed:
Given $(\hat x,\hat t)\in\cD_\textup{SDP}$, if $(\hat x,\hat t)\in\cD$ return $(\hat x,\hat t)$, else decompose $(\hat x,\hat t)$ as a finite convex combination of points $(x_\alpha,t_\alpha)\in\cD_\textup{SDP}$ with $\aff\dim(\cF(x_\alpha))>\aff\dim(\cF(\hat x))$ and recursively compute convex decompositions of $(x_\alpha,t_\alpha)$.

A similar proof gives the following sufficient condition in the context of the SDP tightness result.
\begin{lemma}
\label{lem:applied_easy_part_sdp_tightness}
Suppose Assumptions~\ref{as:gamma_definite}~and~\ref{as:cF_well_defined} hold. Suppose furthermore that:
\begin{equation}
  \label{eq:applied_easy_part_sdp_tightness}
  \parbox{\dimexpr\linewidth-4em}{\strut
    For every optimal $(\hat x,\hat t)\in\cD_\textup{SDP}$ with $\cF(\hat x)$ semidefinite, there exists a point $(x',t')\in\cD_\textup{SDP}$ such that $t'\leq \hat t$ and $\aff\dim(\cF(x'))>\aff\dim(\cF(\hat x))$.
    \strut}
\end{equation}
Then $\Opt=\Opt_\textup{SDP}$.
\end{lemma}
The proof of this statement follows the proof of Lemma~\ref{lem:applied_easy_part_conv_hull} almost exactly and is omitted.

The ``hard part'' of our framework for the convex hull result is to give sufficient conditions for \eqref{eq:applied_easy_part_conv_hull}. We give examples of such conditions in Section~\ref{sec:conv_hull}.
Similarly, the ``hard part'' of our framework for the SDP tightness result is to give sufficient conditions for \eqref{eq:applied_easy_part_sdp_tightness}. We give examples of such conditions in Section~\ref{sec:tightness}. 

\section{Symmetries in QCQPs}
\label{sec:symmetries}

In this section, we examine a parameter $k$ that captures the amount of symmetry present in a QCQP of the form \eqref{eq:qcqp}.

\begin{definition}
The \textit{quadratic eigenvalue multiplicity} of a QCQP of the form \eqref{eq:qcqp} is the largest integer $k$ such that for every $i\in\intset{0,m}$ there exists $\cA_i \in \bb S^n$ for which $A_i = I_k\otimes \cA_i$.
\end{definition}

The quadratic eigenvalue multiplicity $k$ is always at least $1$ as we can write each $A_i$ as $A_i = I_1 \otimes \cA_i$. On the other hand, it is clear that $k$ must be a divisor of $N$. In particular, $k$ is always well defined.

For $\gamma\in\R^m$, we also define $\cA(\gamma) \coloneqq \cA_0 + \sum_{i=1}^m \gamma_i \cA_i$.

\begin{example}
Consider the following optimization problem
\begin{align*}
\inf_{x\in\R^4}\set{-\norm{x}_2^2:\, \begin{array}
	{l}
	x_1^2 - x_2^2 + x_3^2 - x_4^2 -1 \leq 0\\
	2x_1^2 + x_2^2 + 2x_3^2 + x_4^2 -1 \leq 0\\
\end{array}}.
\end{align*}
The quadratic forms in this problem are
\begin{align*}
A_0 = I_2 \otimes \left(\begin{smallmatrix}
	-1 &\\
	& -1
\end{smallmatrix}\right),\qquad
A_1 = I_2 \otimes \left(\begin{smallmatrix}
	1 &\\
	& -1
\end{smallmatrix}\right),\quad\text{and}\quad
A_2 = I_2 \otimes \left(\begin{smallmatrix}
	2 &\\
	& 1
\end{smallmatrix}\right).
\end{align*}
Thus, this QCQP has quadratic eigenvalue multiplicity $k\geq 2$. Recalling that $k$ must be a divisor of $N$ and noting that $A_1$ cannot be written as $A_1 = I_4 \otimes \cA_1$ for any $\cA_1\in\S^1$, we conclude that $k = 2$.\mathprog{\qed}
\end{example}

\begin{remark}\label{rem:computing_k_distinct_eigenvalues}
Suppose we have access to some $\mu\in\R^m$ such that $\cA(\mu)$ has distinct eigenvalues. Then, by simply performing a spectral decomposition of $A(\mu)$ and counting the multiplicities of the eigenvalues, we can correctly output the value $k$.\mathprog{\qed}
\end{remark}

\begin{remark}\label{rem:cstar_algebras}
The quadratic eigenvalue multiplicity can be viewed as a particular \textit{group symmetry} in $\set{A_0,A_1,\dots,A_m}$. Group symmetric SDPs have been studied in more generality with the goal of reducing the size of large SDPs (and in turn their solve-times)~\cite{gatermann2004symmetry,deKlerk2007reduction}. See also~\cite{deKlerk2010exploiting} for an application of such ideas to solving large real-world instances of the quadratic assignment problem.

Specifically, the \textit{Wedderburn decomposition} of the matrix $\C^*$-algebra generated by $\set{A_0,A_1,\dots,A_m}$ plays a prominent role in the analysis of such symmetries. In this setting, our parameter $k$ can be compared to the ``block multiplicity'' of a basic algebra in the Wedderburn decomposition.
This decomposition can be computed efficiently given access to a generic element from the center of the algebra (see \cite{deKlerk2011numerical,gijswijt2010matrix} and references therein).\mathprog{\qed}
\end{remark}

Recall that in Lemma~\ref{lem:semidefinite_shares_zero}, we showed that $\dim(\cV(\cF))\geq 1$ whenever $\cF$ is a semidefinite face of $\Gamma$. The following lemma will show that when the quadratic eigenvalue multiplicity is large, we can in fact lower bound $\dim(\cV(\cF))\geq k$.
This is the main property of the quadratic eigenvalue multiplicity that we will use in Sections~\ref{sec:conv_hull}~and~\ref{sec:tightness}.

\begin{lemma}
\label{lem:V_F_large}
If $\mc F$ is a semidefinite face of $\Gamma$, then $\dim(\mc V(\mc F)) \geq k$.
\end{lemma}
\begin{proof}
By Lemma~\ref{lem:semidefinite_shares_zero}, there exists $\hat v\in \mc V(\mc F)\cap \mb S^{N-1}$. We can write $\hat v$ as the concatenation of $k$-many $n$-dimensional vectors $v_1,\dots,v_k \in\R^n$. Then for $\gamma\in\mc F$,
\begin{align*}
0 = A(\gamma)\hat v = \begin{pmatrix}
	\cA(\gamma) &&&\\
	&\cA(\gamma) &&\\
	&&\ddots &\\
	&&&\cA(\gamma)
\end{pmatrix}\begin{pmatrix}
	v_1\\ v_2 \\ \vdots \\ v_k
\end{pmatrix}
= \begin{pmatrix}
	\cA(\gamma) v_1 \\ \cA(\gamma) v_2 \\\vdots\\ \cA(\gamma) v_k
\end{pmatrix}.
\end{align*}
Hence, $\cA(\gamma) v_i = 0 $ for all $i\in\intset{k}$. As $\hat v\neq 0$, there exists some $i\in\intset{k}$ such that $v_i\neq 0$. Finally, note that for all $y\in\R^k$,
\begin{align*}
A(\gamma) (y\otimes v_i) = (I_k \otimes \cA(\gamma)) (y\otimes v_i) = y\otimes (\cA(\gamma)v_i) = 0.
\end{align*}
In other words, $\R^k \otimes v_i \subseteq \mc V(\mc F)$ and thus $\dim(\mc V(\mc F))\geq k$.
\end{proof}

\begin{remark}
In quadratic matrix programming \cite{beck2007quadratic,beck2012new}, we are asked to optimize
\begin{align}
\label{eq:quadratic_matrix_programming}
\inf_{X\in\R^{n\by k}}\set{\tr(X^\top \cA_0 X)+2\tr(B_0^\top X) + c_0: \, \begin{array}
	{r}
	\tr(X^\top \cA_i X)+2\tr(B_i^\top X) + c_i \leq 0,\quad\\
	\forall i\in\intset{m_I}\\
	\tr(X^\top \cA_i X)+2\tr(B_i^\top X) + c_i = 0,\quad\\
	\forall i\in\intset{m_I+1,m}
\end{array}}, 
\end{align}
where $\cA_i\in\bb S^n$, $B_i\in\R^{n\by k}$ and $c_i\in\R$ for all $i\in\intset{0,m}$. We can transform this program to an equivalent QCQP in the vector variable $x\in\R^{nk}$ by identifying
\begin{align*}
X = \begin{pmatrix}
	x_1 & \dots & x_{(k-1)n + 1}\\
	\vdots & \ddots & \vdots\\
	x_n & \dots & x_{kn}
\end{pmatrix}.
\end{align*}
Then
\begin{align*}
\tr(X^\top \cA_i X)+2\tr(B_i^\top X) + c_i = x^\top \left(I_k\otimes \cA_i\right)x+ 2 b_i^\top x + c_i,
\end{align*}
where, $b_i\in\R^{nk}$ has entries $(b_i)_{(t-1)n + s} = (B_i)_{s,t}$. In particular, the vectorized reformulation of \eqref{eq:quadratic_matrix_programming} has quadratic eigenvalue value multiplicity $k$.\mathprog{\qed}
\end{remark} 

\section{Convex hull results}
\label{sec:conv_hull}
In this section, we present new sufficient conditions for the convex hull result $\mc D_\textup{SDP} = \conv(\mc D)$. We will first analyze the case where the geometry of $\Gamma$ is particularly nice.

\begin{assumption}
\label{as:gamma_polyhedral}
	Assume that $\Gamma$ is polyhedral.\mathprog{\qed}
\end{assumption}

We remark that although Assumption~\ref{as:gamma_polyhedral} is rather restrictive, it is general enough to cover the case where the set of quadratic forms $\set{A_i}_{i\in\intset{0,m}}$ is diagonal or simultaneously diagonalizable---a class of QCQPs which has been studied extensively in the literature (see Section~\ref{subsec:socp} for references).
We will present examples and non-examples of Assumption~\ref{as:gamma_polyhedral} in Sections~\ref{subsec:applications_of_conv_hull}~and~\ref{subsec:socp} and discuss the difficulties in removing this assumption in Section~\ref{subsec:sharpness}. Finally, we will recover weaker results without this assumption in Section~\ref{sec:removing_polyhedral}.

Note that Assumption~\ref{as:gamma_polyhedral} immediately implies Assumption~\ref{as:cF_well_defined} so that we may apply the framework from Section~\ref{sec:framework}.

Our main result in this section is the following theorem.
\begin{theorem}
\label{thm:conv_hull_main}
Suppose Assumptions~\ref{as:gamma_definite}~and~\ref{as:gamma_polyhedral} hold. Furthermore, suppose that for every semidefinite face $\mc F$ of $\Gamma$ we have
\begin{align*}
\dim(\mc V(\mc F)) \geq \aff\dim(\set{b(\gamma):\,\gamma\in\mc F}) + 1.
\end{align*}
Then,
\begin{align*}
\conv(\mc D) = \mc D_\textup{SDP} \quad\text{and}\quad \Opt = \Opt_\textup{SDP}.
\end{align*}
\end{theorem}

As before, the second identity follows immediately from the first identity, thus it suffices to prove only the former. The main effort in this section will be the proof of the following lemma.

\begin{lemma}
\label{lem:pivoting_F}
Suppose Assumptions~\ref{as:gamma_definite}~and~\ref{as:gamma_polyhedral} hold. Furthermore, suppose that for every semidefinite face $\mc F$ of $\Gamma$ we have
\begin{align*}
\dim(\mc V(\mc F)) \geq \aff\dim(\set{b(\gamma):\,\gamma\in\mc F}) + 1.
\end{align*}
Let $(\hat x,\hat t)\in\mc D_\textup{SDP}$ and let $\mc F = \mc F(\hat x)$. If $\mc F$ is a semidefinite face of $\Gamma$, then $(\hat x,\hat t)$ can be written as a convex combination of points $(x_\alpha,t_\alpha)\in\cD_\textup{SDP}$ such that $\aff\dim(\cF(x_\alpha))>\aff\dim(\cF(\hat x))$.
\end{lemma}

The proof of Theorem~\ref{thm:conv_hull_main} follows at once from Lemma~\ref{lem:pivoting_F} and Lemma~\ref{lem:applied_easy_part_conv_hull}.

Before proving Lemma~\ref{lem:pivoting_F}, we introduce some new notation for handling the recessive directions of $\Gamma$ and prove a straightforward lemma about decomposing $\Gamma$. Let
\begin{align*}
\breve A(\gamma) \coloneqq \sum_{i=1}^m \gamma_i A_i,\quad
\breve b(\gamma) \coloneqq \sum_{i=1}^m \gamma_i b_i,\quad
\breve c(\gamma) \coloneqq \sum_{i=1}^m \gamma_i c_i,\quad
\breve q(\gamma,x) \coloneqq \sum_{i=1}^m \gamma_i q_i(x).
\end{align*}

\begin{lemma}
\label{lem:decompose_gamma}
Suppose Assumption~\ref{as:gamma_polyhedral} holds. Then $\Gamma$ can be written as
\begin{align*}
\Gamma = \Gamma_e + \cone(\Gamma_r)
\end{align*}
where both $\Gamma_e$ and $\Gamma_r$ are polytopes. Here, $\Gamma_r$ may be the trivial set $\set{0}$.
Furthermore, for $\hat x\in\R^N$ such that $\sup_{\gamma\in\Gamma} q(\gamma,\hat x)$ is finite, we have
\begin{align*}
\mc F(\hat x) = \mc F_e(\hat x) + \cone(\mc F_r(\hat x))
\end{align*}
where $\mc F_e(\hat x)$ is the face of $\Gamma_e$ maximizing $q(\gamma,\hat x)$ and $\mc F_r(\hat x)$ is the face of $\Gamma_e$ satisfying $\breve q(\gamma, \hat x) = 0$.
\end{lemma}
\begin{proof}
This follows immediately from the Minkowski-Weyl Theorem and noting that $\breve q(\gamma_r,\hat x)\leq 0$ for all $\gamma_r\in\Gamma_r$ when $\sup_{\gamma\in\Gamma} q(\gamma,\hat x)$ is finite.
\end{proof}

\begin{proof}[Proof of Lemma~\ref{lem:pivoting_F}]
Without loss of generality, we may assume that $\sup_{\gamma\in\Gamma} q(\gamma,\hat x) = 2\hat t$. Otherwise,  we can decrease $\hat t$ and note that $\mc D$ is closed upwards in the $t$-direction. In particular, we have that $q(\gamma,\hat x)$ achieves the value $2\hat t$ on $\mc F$.

We claim that the following system in variables $v$ and $s$
\begin{align*}
	\begin{cases}
		\ip{b(\gamma), v} = s,\,\forall \gamma\in \mc F\\
		v\in\mc V(\mc F),\,s\in\R
	\end{cases}
\end{align*}
has a nonzero solution. Indeed, we may replace the constraint $\ip{b(\gamma), v} = s,\,\forall \gamma\in \mc F$ with at most
\begin{align*}
\aff\dim(\set{b(\gamma):\, \gamma\in\mc F}) + 1 \leq \dim(\mc V(\mc F))
\end{align*}
homogeneous linear equalities in the variables $v$ and $s$.
The claim then follows by noting that the equivalent system is an under-constrained homogeneous system of linear equalities and thus has a nonzero solution $(v,s)$. It is easy to verify that $v\neq 0$, hence by scaling we may take $v\in\mb S^{N-1}$. In the remainder of the proof, let $v\in\mc V(\mc F)\cap\mb S^{N-1}$ and $s\in \R$ denote a solution pair to the above system.

Apply Lemma~\ref{lem:decompose_gamma} to decompose $\Gamma = \Gamma_e +\cone(\Gamma_r)$ and $\mc F = \mc F_e + \cone(\mc F_r)$.

We will modify $(\hat x,\hat t)$ in the
$(v, s)$
direction. For $\alpha\in\R$, we define
\begin{align*}
(x_\alpha,t_\alpha)\coloneqq \left(\hat x + \alpha v,\, \hat t + \alpha s\right).
\end{align*}

First, for any fixed $\gamma_f\in\mc F$, we consider how $q(\gamma_f,x_\alpha) - 2t_\alpha$ changes with $\alpha$. We can expand
\begin{align*}
q(\gamma_f,x_\alpha) - 2t_\alpha &= \left(q(\gamma_f,\hat x) - 2\hat t\right) + 2\alpha \left(\hat x^\top A(\gamma_f)v+ b(\gamma_f)^\top v - s\right) + \alpha^2 v^\top A(\gamma_f) v\\
&= q(\gamma_f,\hat x) - 2\hat t\\
&= 0,
\end{align*}
where the second line follows as $A(\gamma_f)v = 0$ (recall $v\in\mc V(\mc F)$) and $b(\gamma_f)^\top v = s$ for all $\gamma_f\in\mc F$, and 
the third line follows as $q(\gamma_f, \hat x) = 2\hat t$ for all $\gamma_f\in\mc F$. 
Now consider any $\gamma_e\in\mc F_e$ and $\gamma_r\in\mc F_r$. Note that $\gamma_e$ and $\gamma_e+\gamma_r$ both lie in $\mc F$. Then by the above calculation, both $\alpha\mapsto q(\gamma_e, x_\alpha)-2 t_\alpha$ and $\alpha\mapsto q(\gamma_e+\gamma_r, x_\alpha)-2 t_\alpha$ are identically zero. In particular, we also have that $\alpha\mapsto \breve q(\gamma_r, x_\alpha) =q(\gamma_e+\gamma_r,x_\alpha) - q(\gamma_e,x_\alpha)= 0$ is identically zero.

On the other hand, for $\gamma_e\in\Gamma_e\setminus \mc F_e$, we can expand
\begin{align*}
q(\gamma_e,x_\alpha) - 2t_\alpha &= \left(q(\gamma_e,\hat x) - 2\hat t\right) + 2\alpha \left(\hat x^\top A(\gamma_e)v+ b(\gamma_e)^\top v - s\right) + \alpha^2 v^\top A(\gamma_e) v,
\end{align*}
and note that  $v^\top A(\gamma_e)v\geq0$ holds because $A(\gamma_e)$ is positive semidefinite. Hence, for $\gamma_e\in\Gamma_e\setminus \mc F_e$, we have that $\alpha\mapsto q(\gamma_e,x_\alpha) - 2t_\alpha$ is a (possibly non-strictly) convex quadratic function taking the value $q(\gamma_e,\hat x) - 2\hat t < 0$ at $\alpha = 0$ (the strict inequality here follows from the fact that $\gamma_e\in\Gamma_e\setminus\mc F_e$).

Similarly, for $\gamma_r\in \Gamma_r\setminus \mc F_r$, we can expand
\begin{align*}
\breve q(\gamma_r,x_\alpha) &= \breve q(\gamma_r,\hat x) + 2\alpha \left(\hat x^\top \breve {A}(\gamma_r)v + \breve b(\gamma_r)^\top v \right) + \alpha^2 v^\top \breve{A}(\gamma_r) v.
\end{align*}
Note that $\breve {A}(\gamma)\succeq 0$ for all $\gamma\in\Gamma_r$. Hence, for $\gamma_r\in\Gamma_r\setminus\mc F_r$, we have that $\alpha\mapsto \breve q(\gamma_r,x_\alpha)$ is a (possibly non-strictly) convex quadratic function taking the value $\breve q(\gamma_r, \hat x)<0$ at $\alpha = 0$ (the strict inequality here follows from the fact that $\gamma_r\in\Gamma_r\setminus\mc F_r$).

We have shown that the following finite set of univariate quadratic functions in $\alpha$,
\begin{align*}
\mc Q \coloneqq \bigg(\set{q(\gamma_e,x_\alpha) - 2 t_\alpha :\, \gamma_e\in\extr(\Gamma_e)} \cup \set{\breve q(\gamma_r,x_\alpha):\, \gamma_r\in\extr(\Gamma_r)} \bigg) \setminus \set{0},
\end{align*}
consists of (possibly non-strictly) convex quadratic functions which are negative at $\alpha = 0$. The finiteness of this set follows from the assumption that $\Gamma$ is polyhedral.

We claim that there exists a quadratic function in $\mc Q$ which is strictly convex:  Note $\gamma^*$ from Assumption~\ref{as:gamma_definite} satisfies $\gamma^*\in\Gamma$. Thus, we can decompose $\gamma^* = \gamma_e + \alpha \gamma_r$ for $\gamma_e\in\Gamma_e$, $\gamma_r\in\Gamma_r$, and $\alpha \geq 0$. Then, 
\begin{align*}
0<v^\top A(\gamma^*) v = \left[v^\top A(\gamma_e)v\right] + \alpha\left[v^\top \breve{A}(\gamma_r)v\right].
\end{align*}
Hence, one of the square-bracketed terms must be positive. The claim then follows by linearity in $\gamma$ of the functions $\gamma\mapsto v^\top A(\gamma)v$ and $\gamma\mapsto v^\top \breve{A}(\gamma)v$.

As $\mc Q$ is a finite set by Assumption~\ref{as:gamma_polyhedral}, there exists an $\alpha_+ >0$ such that $q(\alpha_+)\leq 0$ for all $q\in\mc Q$ with at least one equality. Then because $\Gamma_e=\conv(\extr(\Gamma_e))$ and $\Gamma_r=\conv(\extr(\Gamma_r))$, we have $q(\gamma_e,x_{\alpha_+}) \leq 2t_{\alpha_+}$ for all $\gamma_e\in\Gamma_e$ and $\breve q(\gamma_r,x_{\alpha_+})\leq 0$ for all $\gamma_r\in\Gamma_r$. Thus, $(x_{\alpha_+},t_{\alpha_+})\in\mc D_\textup{SDP}$. 

It remains to show that $\aff\dim(\mc F(x_{\alpha_+}))>\aff\dim(\mc F(\hat x))$.
The discussion in the previous paragraph implies that $\sup_{\gamma\in\Gamma}q(\gamma,x_{\alpha_+}) \leq 2t_{\alpha_+}$. This value is achieved by $\gamma_f\in \mc F(\hat x)$: Note $q(\gamma_f,x_{\alpha_+}) - 2t_{\alpha_+} = q(\gamma_f,\hat x) - 2\hat t = 0$. In particular, $\mc F(\hat x)\subseteq \mc F(x_{\alpha_+})$.
Thus, it suffices to show that there exists $\gamma_+ \in \mc F(x_{\alpha_+})\setminus \mc F(\hat x)$.

Suppose the quadratic function in $\mc Q$ with $\alpha_+$ as a root is of the form $q(\gamma_+,x_\alpha)-2t_\alpha$. Then $\gamma_+\in\mc F(x_{\alpha_+})$ as $q(\gamma_+,x_{\alpha_+}) -2t_{\alpha_+} = 0$.
On the other hand, $\gamma_+\notin\mc F(\hat x)$ by the construction of $\mc Q$.

Suppose the quadratic function in $\mc Q$ with $\alpha_+$ as a root is of the form $\breve q(\gamma_r,x_\alpha)$. Select any $\gamma_f\in\mc F(\hat x)$ and recall that $q(\gamma_f, x_\alpha) - 2t_\alpha$ is identically zero as an expression in $\alpha$. Define $\gamma_+ = \gamma_f +\gamma_r$. Then,
\begin{align*}
q(\gamma_+,x_{\alpha_+}) -2 t_{\alpha_+} = \left(q(\gamma_f,x_{\alpha_+}) -2 t_{\alpha_+}\right) + \breve q(\gamma_r,x_{\alpha_+}) = 0
\end{align*}
and hence $\gamma_+\in\mc F(x_{\alpha_+})$. On the other hand, $\breve q(\gamma_r, \hat x) < 0$ by the construction of $\mc Q$. In particular,
\begin{align*}
q(\gamma_+,\hat x) - 2\hat t = \left(q(\gamma_f,\hat x) - 2\hat t\right) + \breve q(\gamma_r, \hat x) <0
\end{align*}
and thus $\gamma_+\notin\mc F(\hat x)$.

The existence of an $\alpha_-<0$ satisfying the same properties is proved analogously. Then we may write $(\hat x,\hat t)$ as a convex combination of $(x_{\alpha_+}, t_{\alpha_+})$ and $(x_{\alpha_-}, t_{\alpha_-})$.
\end{proof}

The next theorem follows as a corollary to Theorem~\ref{thm:conv_hull_main}.

\begin{theorem}
\label{thm:conv_hull_symmetries}
Suppose Assumptions~\ref{as:gamma_definite}~and~\ref{as:gamma_polyhedral} hold. Furthermore, suppose that for every semidefinite face $\cF$ of $\Gamma$ we have
\begin{align*}
k\geq \aff\dim(\set{b(\gamma):\, \gamma\in\mc F}) + 1.
\end{align*}
Then,
\begin{align*}
\conv(\mc D) = \mc D_\textup{SDP}
\quad\text{and}\quad
\Opt &= \Opt_\textup{SDP}.
\end{align*}
\end{theorem}
\begin{proof}
This theorem follows from Lemma~\ref{lem:V_F_large} and Theorem~\ref{thm:conv_hull_main}.
\end{proof}

\begin{remark}\label{rem:soc_representability}
We remark that when $\Gamma$ is polyhedral (Assumption~\ref{as:gamma_polyhedral}), the set $\mc D_\textup{SDP}$ is actually SOC representable: By Lemmas~\ref{lemma:sdp_in_terms_of_Gamma}~and~\ref{lem:decompose_gamma} we can write
\begin{align*}
\mc D_\textup{SDP} &= \set{(x,t):\, \sup_{\gamma\in\Gamma} q(\gamma,x)\leq 2t}\\
&=  \set{(x,t):\,\begin{array}
	{l}
	q(\gamma_e,x) \leq 2t ,\,\forall \gamma_e\in\extr(\Gamma_e)\\
	\breve q(\gamma_f,x)\leq 0 ,\,\forall \gamma_f\in\extr(\Gamma_r)
\end{array}}.
\end{align*}
In other words, $\mc D_\textup{SDP}$ is defined by finitely many convex quadratic inequalities. In particular, the assumptions of Theorem~\ref{thm:conv_hull_main}~and~\ref{thm:conv_hull_symmetries} imply that $\conv(\mc D)$ is SOC representable.\mathprog{\qed}
\end{remark}

\subsection{Applications of Theorems~\ref{thm:conv_hull_main}~and~\ref{thm:conv_hull_symmetries}}
\label{subsec:applications_of_conv_hull}

We now state some classes of problems where the assumptions of Theorems~\ref{thm:conv_hull_main}~and~\ref{thm:conv_hull_symmetries} hold.

The most basic setup we can cover via these theorems is the case of convex quadratic programs.
\begin{corollary}
\label{cor:conv_qcqps}
Suppose Assumption~\ref{as:gamma_definite} holds. If $A_0\succ 0$, $m_E = 0$ and $A_i\succeq 0$ for all $i\in\intset{m_I}$, then
\begin{align*}
\conv(\mc D) = \mc D_\textup{SDP}
\quad\text{and}\quad
\Opt = \Opt_\textup{SDP}.
\end{align*}
\end{corollary}
\begin{proof}
Assumption~\ref{as:gamma_polyhedral} holds in this case as
\begin{align*}
\Gamma &= \set{\gamma\in\R^m:\, \begin{array}
	{l}
	A(\gamma)\succeq 0\\
	\gamma\geq 0
\end{array}} = \set{\gamma\in\R^m:\, \gamma\geq 0}.
\end{align*}
Furthermore, each face of $\Gamma$ contains the origin. Thus noting that $A(0)=A_0\succ 0$ is positive definite, we conclude that $\Gamma$ does not have any semidefinite face. This allows us to apply Theorem~\ref{thm:conv_hull_symmetries}.
\end{proof}

\begin{remark}
It is possible to apply a standard limit argument (see for example \cite{burer2019exact}) to handle additionally the case where $A_0$ is only positive semidefinite.\mathprog{\qed}
\end{remark}

Next, we discuss a number of results on TRS and GTRS.
\begin{corollary}\label{cor:m=1}
Suppose $m = 1$ and Assumption~\ref{as:gamma_definite} holds. Then,
\begin{align*}
\conv(\mc D) = \mc D_\textup{SDP}
\quad\text{and}\quad
\Opt = \Opt_\textup{SDP}.
\end{align*}
\end{corollary}
\begin{proof}
The set $\Gamma$ will either be a bounded interval $[\gamma_1,\gamma_2]$, a semi-infinite interval $[\gamma_1,\infty)$, or the entire line $(-\infty,\infty)$. In all three cases, $\Gamma$ is polyhedral and Assumption~\ref{as:gamma_polyhedral} holds.

By Observation~\ref{obs:aff_dim_m_is_definite}, any semidefinite face of $\Gamma$ must have affine dimension at most $m - 1 = 0$.
In particular $\aff\dim(\set{b(\gamma):\, \gamma\in\mc F})=0$ and the assumption on the quadratic eigenvalue multiplicity in Theorem~\ref{thm:conv_hull_symmetries} holds as $k$ is always at least $1$.
\end{proof}

Corollary~\ref{cor:m=1} in particular recovers the well-known results associated with the epigraph set of the TRS\footnote{
Corollary~\ref{cor:m=1} fails to fully recover
\cite[Theorem 13]{hoNguyen2017second}. Indeed, \cite[Theorem 13]{hoNguyen2017second} also gives a description of the convex hull of the epigraph of the TRS with an additional conic constraint under some assumptions. We do not consider these additional conic constraints in our setup.
}
and the GTRS (see \cite[Theorem 13]{hoNguyen2017second} and \cite[Theorems 1 and 2]{wang2019generalized}).

\begin{corollary}\label{cor:b_i=0}
Suppose Assumptions~\ref{as:gamma_definite}~and~\ref{as:gamma_polyhedral} hold. If $b_i= 0$ for all $i\in\intset{m}$, then
\begin{align*}
\conv(\mc D) = \mc D_\textup{SDP}
\quad{and}\quad
\Opt &= \Opt_\textup{SDP}.
\end{align*}
\end{corollary}
\begin{proof}
Note that $b(\gamma) = b_0 + \sum_{i=1}^m \gamma_i b_i = b_0$ for any $\gamma\in\R^m$. Thus, for any face $\mc F$ of $\Gamma$, we have
\begin{align*}
\aff\dim(\set{b(\gamma):\, \gamma\in\mc F}) + 1 = \aff\dim(\set{b_0})+1 = 1.
\end{align*}
In particular, the assumptions on the quadratic eigenvalue multiplicity in Theorem~\ref{thm:conv_hull_symmetries} holds as $k$ is always at least $1$.
\end{proof}

\begin{example}
\label{ex:b_i=0}
Consider the following optimization problem
\begin{align*}
\inf_{x\in\R^2} \set{x_1^2 +x_2^2 +10x_1 :\, \begin{array}
	{l}
	x_1^2 - x_2^2 -5 \leq 0\\
	-x_1^2 + x_2^2 -50 \leq 0
\end{array}}.
\end{align*}
We check that the conditions of Corollary~\ref{cor:b_i=0} hold. Assumption~\ref{as:gamma_definite} holds as $A(0) = A_0 = I\succ 0$ and $x= 0$ is feasible. Next, Assumption~\ref{as:gamma_polyhedral} holds as
\begin{align*}
\Gamma &= \set{\gamma\in\R^2:\, \begin{array}
	{l}
	1+\gamma_1-\gamma_2\geq 0\\
	1-\gamma_1+\gamma_2\geq 0\\
	\gamma\geq 0
\end{array}}.
\end{align*}
One can verify that
\begin{align*}
\Gamma = \conv\left(\set{(0,0), (1,0), (0,1)}\right) + \cone(\set{1,1}).
\end{align*}
Finally, we note that $b_1 =b_2=0$. Hence, Corollary~\ref{cor:b_i=0} and Remark~\ref{rem:soc_representability} imply that
\begin{align*}
\conv(\mc D) = \mc D_\textup{SDP} = \set{(x,t) :\, \begin{array}
	{l}
	x_1^2+x_2^2 +10 x_1\leq 2t\\
	2x_1^2 +10x_1 -5 \leq 2t\\
	2x_2^2 +10x_1 - 50 \leq 2t
\end{array}}.
\end{align*}
We plot $\mc D$ and $\conv(\mc D)=\mc D_\textup{SDP}$ in Figure~\ref{fig:example_convex_hull}.\mathprog{\qed}
\end{example}
\begin{figure}
  \centering
    \includegraphics[width=0.4\textwidth]{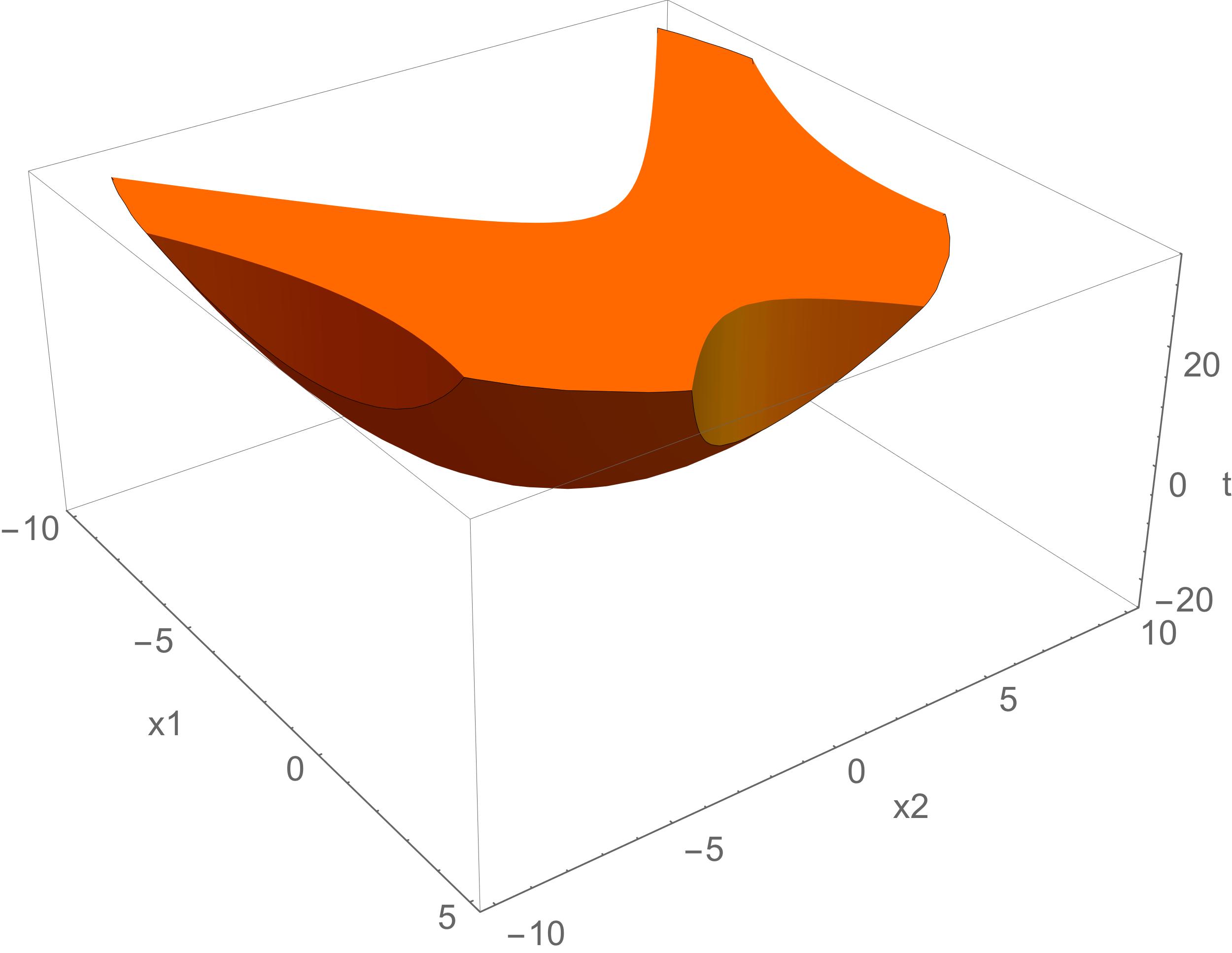}\qquad
    \includegraphics[width=0.4\textwidth]{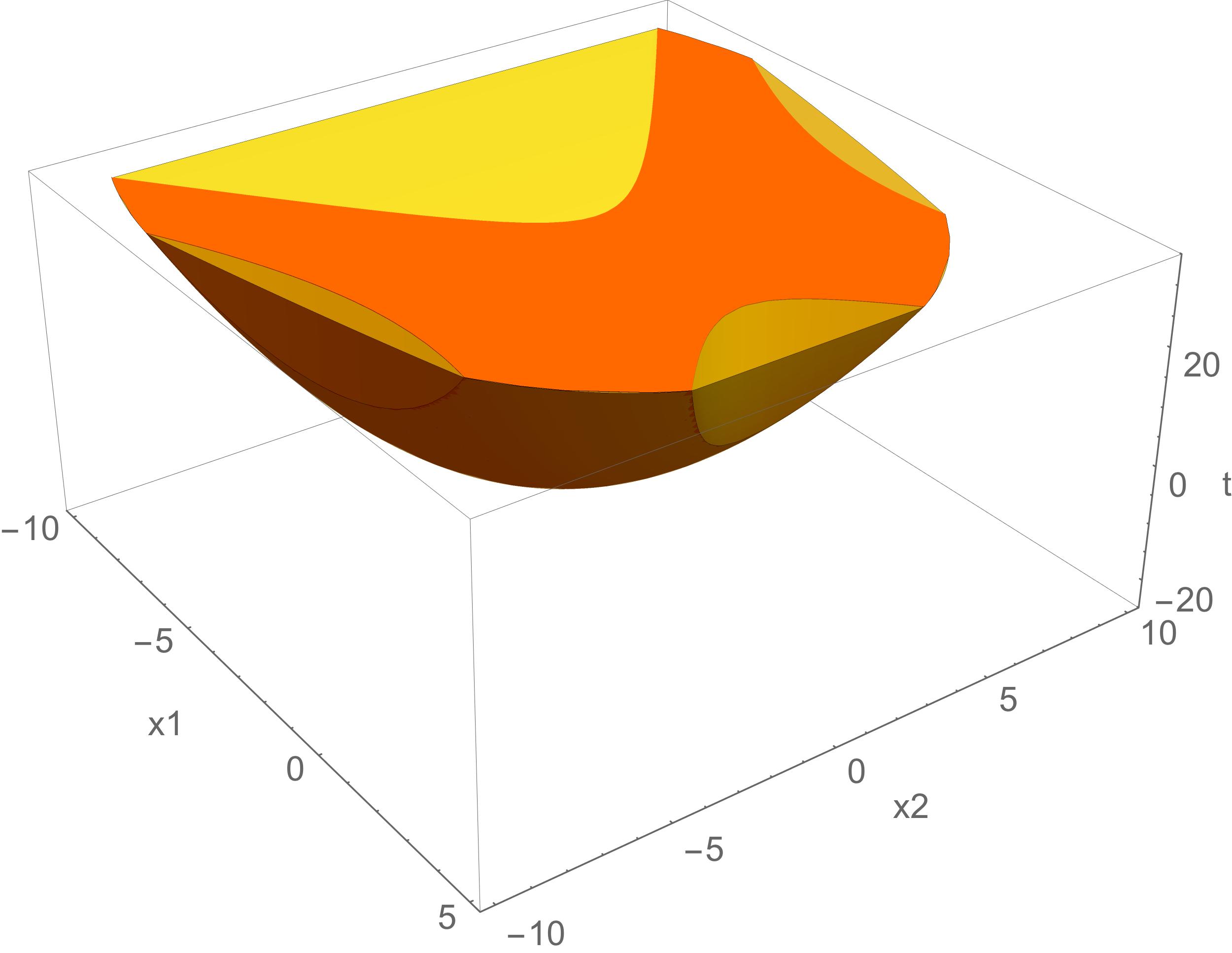}
	\caption{The sets $\mc D$ (in orange) and $\conv(\mc D)$ (in yellow) from Example~\ref{ex:b_i=0}}
	\label{fig:example_convex_hull}
\end{figure}

\begin{remark}[Joint zero of a finite set of quadratic forms]\label{ex:Barvinok}
\citet{barvinok1993feasibility} shows that one can decide in polynomial time (in $N$) whether a constant number, $m_E$, of quadratic forms $\set{A_i}_{i\in\intset{m_E}}$ has a joint nontrivial zero. That is, whether the system $x^\top A_i x=0$ for $i\in\intset{m_E}$ and $x^\top x = 1$ is feasible. We can recast this as asking whether the following optimization problem
\begin{align*}
\min_x\set{-x^\top x:~  \begin{array}
	{l}
	x^\top x\leq 1\\
	x^\top A_i x=0,\forall i\in\intset{m_E}
\end{array}}
\end{align*}
has objective value $-1$ or $0$.

Thus, the feasibility problem studied in \cite{barvinok1993feasibility} reduces to a QCQP of the form we study in this paper.
Note that Assumption~\ref{as:gamma_definite} for a QCQP of this form holds, for example, by taking $\gamma^* = 2 e_1$ so that $A(\gamma^*) = -I + 2I\succ 0$ and noting that $x = 0$ is a feasible solution to this QCQP.
Then when $\Gamma$ is polyhedral (Assumption~\ref{as:gamma_polyhedral}), Corollary~\ref{cor:b_i=0} implies that the feasibility problem (in even a variable number of quadratic forms) can be decided using a semidefinite programming approach.
Nevertheless, Assumption~\ref{as:gamma_polyhedral} may not necessarily hold, and thus Corollary~\ref{cor:b_i=0}  does not recover the full result of \cite{barvinok1993feasibility}.\mathprog{\qed}
\end{remark}

\begin{corollary}
\label{cor:identity_matrices}
Suppose Assumption~\ref{as:gamma_definite} holds and for every $i\in\intset{0,m}$, there exists $\alpha_i$ such that $A_i = \alpha_i I_N$. If $m\leq N$, then
\begin{align*}
\conv(\mc D) = \mc D_\textup{SDP}
\quad{and}\quad
\Opt &= \Opt_\textup{SDP}.
\end{align*}
\end{corollary}
\begin{proof}
Assumption~\ref{as:gamma_polyhedral} holds in this case as
\begin{align*}
\Gamma &\coloneqq \set{\gamma\in\R^m:\, \begin{array}
	{l}
	A(\gamma)\succeq 0\\
	\gamma_i\geq 0,\,\forall i\in\intset{m_I}
\end{array}}= \set{\gamma\in\R^m:\, \begin{array}
	{l}
	\alpha_0 +\sum_{i=1}^m \gamma_i \alpha_i\geq 0\\
	\gamma_i\geq 0,\,\forall i\in\intset{m_I}
\end{array}}
\end{align*}
is defined by $m_I+1$ linear inequalities.

As each $A_i = \alpha_i I_N$, we have that the quadratic eigenvalue multiplicity satisfies $k = N$. By Observation~\ref{obs:aff_dim_m_is_definite}, any semidefinite face of $\Gamma$ must have affine dimension at most $m - 1$.
In particular $\aff\dim(\set{b(\gamma):\, \gamma\in\mc F}) + 1 \leq m$ and the assumption on the quadratic eigenvalue multiplicity in Theorem~\ref{thm:conv_hull_symmetries} holds as $k= N \geq m$. The final inequality $N\geq m$ holds by the assumptions of the corollary.
\end{proof}

\begin{remark}
Consider the problem of finding the distance between the origin $0\in\R^N$ and a piece of ``Swiss cheese'' $C\subseteq \R^N$. We will assume that $C$ is nonempty and defined as
\begin{align*}
C = \set{x\in\R^N :\, \begin{array}
	{l}
	\norm{x - y_i} \leq s_i ,\,\forall i\in\intset{m_1}\\
	\norm{x - z_i} \geq t_i ,\,\forall i\in\intset{m_2}\\
	\ip{x,b_i} \geq c_i ,\,\forall i\in\intset{m_3}
\end{array}},
\end{align*}
where $y_i, z_i, b_i\in\R^N$ and $s_i,t_i,c_i\in\R$ are arbitrary.
In words, $C$ is defined by $m_1$-many ``inside-ball'' constraints, $m_2$-many ``outside-ball'' constraints, and $m_3$-many linear inequalities. Note that each of these constraints may be written as a quadratic inequality with a quadratic form $I$, $-I$, or $0$. In particular, Corollary~\ref{cor:identity_matrices} implies that if $m_1+m_2+m_3\leq N$, then the value
\begin{align*}
\inf_{x\in\R^N}\set{\norm{x}^2 :\, x\in C}
\end{align*}
may be computed using the standard SDP relaxation of the problem.

\citet{bienstock2014polynomial} give sufficient conditions under which a related problem
\begin{align*}
\inf_{x\in\R^N}\set{q_0(x) :\, x\in C},
\end{align*}
is polynomial-time solvable. Here, $q_0:\R^N\to\R$ is an arbitrary quadratic function but $m_1$ and $m_2$ are constant.
Specifically, they devise an enumerative algorithm for problems of this form and prove its correctness under different assumptions.
In contrast, our work deals only with the standard SDP relaxation and does not assume that the number of quadratic forms is constant.

\citet{yang2018quadratic} consider QCQPs with additional ``hollow'' type constraints. Formally, they consider a QCQP with domain $\cG\coloneqq F\setminus \bigcup_\alpha \inter(E_\alpha)$ where $F$ is a quadratically constrained domain and $\set{E_\alpha}$ is a finite collection of \textit{non-intersecting} ellipsoids completely contained within $F$. They show that if the SDP relaxation for a QCQP over the domain $F$ is exact, then the SDP relaxation strengthened by additional linear constraints is exact for the same QCQP over the domain $\cG$. In contrast, Corollary~\ref{cor:identity_matrices} makes no assumption on how the constraints defining $\cG$ intersect but deals only with linearly many (in the dimension) spherical constraints.\mathprog{\qed}
\end{remark}
 \subsection{SD-QCQPs and the polyhedrality assumption}\label{subsec:socp}

A natural class of QCQPs where Assumption~\ref{as:gamma_polyhedral} is immediately satisfied is the class of simultaneously diagonalizable QCQPs (SD-QCQPs) (see Definition~\ref{def:sd_qcqp} below).
In this section, we first discuss how the simultaneously diagonalizable (SD) assumption relates to the polyhedrality assumption.
Then, in Section~\ref{subsec:socp_rel_SD-QCQPs}, we show that under the SD assumption, the standard SDP relaxation is in fact equivalent to the lifted SOCP relaxation (both in terms of optimal value and projected epigraph).
Consequently, our framework automatically generates sufficient conditions for SOCP-based tightness and convex hull results.
Such sufficient conditions have been studied in the literature and we will compare our conditions with sufficient conditions proposed by~\citet{benTal2014hidden}~and~\citet{locatelli2016exactness} in Section~\ref{subsec:literature_comparison}.

Recall the following definition.
\begin{definition}
A set of matrices $\set{A_i}_{i\in\intset{0,m}}\subseteq \Se^N$ is said to be \emph{simultaneously diagonalizable} (SD) if there exists an invertible matrix $U\in \R^{N\by N}$ such that the set $\big\{U^\top A_i U\big\}_{i\in\intset{0,m}}$ consists of diagonal matrices.
\end{definition}
We note that this condition, sometimes referred to as simultaneously diagonalizable \emph{by congruence}, is weaker than the notion of being simultaneously diagonalizable \emph{by similarity} which further requires that $U$ be an orthonormal matrix.

\begin{definition}
\label{def:sd_qcqp}
A \emph{simultaneously diagonalizable QCQP} (SD-QCQP) is a QCQP of the form~\eqref{eq:qcqp} where $\set{A_i}_{i\in\intset{0,m}}$ is SD.
\end{definition}

\begin{lemma}
\label{lem:SD_implies_polyhedral}
For any SD-QCQP, we have that $\Gamma$ is polyhedral.
\end{lemma}
\begin{proof}
Let $U\in\R^{N\by N}$ be an invertible matrix such that $U^\top A_i U = \Lambda_i$ is diagonal for each $i\in\intset{0,m}$.
Note that $A(\gamma)\succeq 0$ if and only if $U^\top A(\gamma)U \succeq 0$ if and only if $\Lambda_0 + \sum_{i=1}^m \gamma_i \Lambda_i \succeq 0$. It is clear that
\begin{align*}
\Gamma = \set{\gamma\in\R^m:\, \begin{array}
	{l}
	\Lambda_0 + \sum_{i=1}^m \gamma_i\Lambda_i \succeq 0 \\
	\gamma_i\geq 0 ,\,\forall i \in\intset{m_I}
\end{array}}
\end{align*}
is polyhedral.
\end{proof}

The following example shows that changing a given constraint in a QCQP from an inequality into an equality constraint can alter whether $\Gamma$ is polyhedral or not.
As a consequence, we will deduce by Lemma~\ref{lem:SD_implies_polyhedral} that Assumption~\ref{as:gamma_polyhedral} is strictly weaker than the simultaneous diagonalizability assumption. 
\begin{example}
\label{ex:Gamma_polyhedral}
Consider the matrices
\begin{align*}
A_0 = \begin{pmatrix}
	1 & &\\
	& \sqrt{2} & 0\\
	& 0& \sqrt{2}
\end{pmatrix},
\quad
A_1 = \begin{pmatrix}
	-1 &&\\
	&1&1\\
	&1&-1
\end{pmatrix},
\quad
A_2 = \begin{pmatrix}
	-1 &&\\
	&1&-1\\
	&-1 & -1
\end{pmatrix}.
\end{align*}
Note that $A(\gamma)\succeq 0$ if and only if each of its two blocks are positive semidefinite. Recall that a $2\by 2$ matrix is positive semidefinite if and only if both its trace and determinant are nonnegative.

Suppose first that $A_1$ and $A_2$ correspond to equality constraints. Then
\begin{align*}
\Gamma &= \set{\gamma\in\R^2 :\, \begin{array}
	{l}
	1 - \gamma_1 - \gamma_2 \geq 0 \\
	(\sqrt{2}+(\gamma_1 +\gamma_2))(\sqrt{2}-(\gamma_1+\gamma_2))-(\gamma_1-\gamma_2)^2\geq 0
\end{array}}\\
&= \set{\gamma\in\R^2 :\, \begin{array}
	{l}
	\gamma_1+\gamma_2 \leq 1 \\
	2 - (\gamma_1+\gamma_2)^2 -(\gamma_1-\gamma_2)^2 \geq 0
\end{array}}\\
&= \set{\gamma\in\R^2 :\, \begin{array}
	{l}
	\gamma_1+\gamma_2 \leq 1 \\
	\gamma_1^2 + \gamma_2^2 \leq 1
\end{array}}.
\end{align*}
is not polyhedral (see Figure~\ref{fig:example_Gamma_polyhedral} left). In particular by Lemma~\ref{lem:SD_implies_polyhedral}, we deduce that the set $\set{A_0,A_1,A_2}$ is not simultaneously diagonalizable.

Now suppose that $A_1$ and $A_2$ correspond to inequality constraints. Then
\begin{align*}
\Gamma &= \set{\gamma\in\R^2 :\, \begin{array}
	{l}
	\gamma_1+\gamma_2 \leq 1 \\
	\gamma_1^2 + \gamma_2^2 \leq 1\\
	\gamma\geq 0
\end{array}} = \set{\gamma\in\R^2 :\, \begin{array}
	{l}
	\gamma_1+\gamma_2 \leq 1 \\
	\gamma\geq 0
\end{array}}
\end{align*}
is polyhedral (see Figure~\ref{fig:example_Gamma_polyhedral} right). Thus, we have constructed an example where the set $\set{A_0, A_1, A_2}$ is not simultaneously diagonalizable but $\Gamma$ is polyhedral. We deduce that Assumption~\ref{as:gamma_polyhedral} is strictly weaker than the simultaneous diagonalizability assumption.\mathprog{\qed}
\end{example}

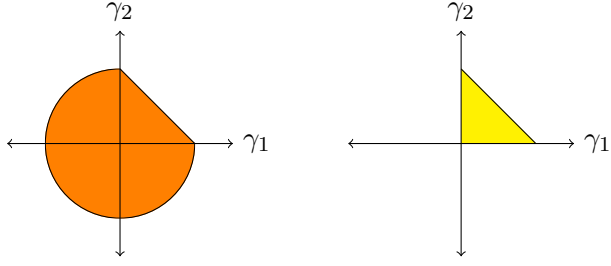
\begin{figure}
 	\centering
\begin{tikzpicture}
	\begin{scope}
		\clip (1,0) -- (0,1) -- (-1,1) -- (-1,-1) -- (1,-1) -- cycle;
		\fill [fill = orange] (0,0) circle (1cm);
		\draw (0,0) circle (1cm-0.5\pgflinewidth);
	\end{scope}
	\draw[-] (0,1cm-0.5\pgflinewidth) -- (1cm-0.5\pgflinewidth,0);

	\draw[<->] (-1.5,0) -- (1.5,0);
	\draw[<->] (0,-1.5) -- (0,1.5);

	\node[right] (g1) at (1.5,0) {$\gamma_1$};
	\node[above] (g2) at (0,1.5) {$\gamma_2$};
\end{tikzpicture}
\qquad
\begin{tikzpicture}
	\fill [fill = yellow] (0,0) -- (1,0) -- (0,1) -- cycle;
	\draw[-] (0,1cm-0.5\pgflinewidth) -- (1cm-0.5\pgflinewidth,0);

	\draw[<->] (-1.5,0) -- (1.5,0);
	\draw[<->] (0,-1.5) -- (0,1.5);

	\node[right] (g1) at (1.5,0) {$\gamma_1$};
	\node[above] (g2) at (0,1.5) {$\gamma_2$};
\end{tikzpicture}
\caption{The set $\Gamma$ with equality (orange) and inequality (yellow) constraints from Example~\ref{ex:Gamma_polyhedral}}
	\label{fig:example_Gamma_polyhedral}
\end{figure}

\begin{remark}\label{rem:polyhedral_not_checkable}
\citet{ramana1997polyhedra} showed that deciding whether a given spectrahedron is polyhedral is coNP-hard. In particular, it is coNP-hard to decide whether Assumption~\ref{as:gamma_polyhedral} holds in general. 
Nevertheless, it is possible to prove that this assumption holds for specific classes of interesting QCQPs (for example see Corollaries~\ref{cor:m=1}~and~\ref{cor:identity_matrices}).\mathprog{\qed}
\end{remark}

\subsubsection{The equivalence of SDP and SOCP relaxations of SD-QCQPs}
\label{subsec:socp_rel_SD-QCQPs}
Given an SD-QCQP and the invertible matrix $U$, we may perform a change of variables to arrive at a diagonal QCQP, i.e., a QCQP of the form~\eqref{eq:qcqp} where each $A_i$ is diagonal. In the remainder of this section, we assume that we have already made this change of variables and are left with
\begin{align}
\label{eq:diagonal_qcqp}
\inf_{x\in\R^N} \set{q_0(x) :\, \begin{array}
	{l}
	q_i(x) \leq 0 ,\,\forall i\in\intset{m_I}\\
	q_i(x) = 0,\,\forall i\in\intset{m_I +1, m}
\end{array}},
\end{align}
where $q_i(x) = \ip{a_i, x^2} + 2 \ip{b_i,x} + c_i$, $a_i\in\R^N$, $b_i\in\R^N$, and $c_i\in \R$ for each $i\in\intset{0,m}$.  Here, $x^2\in\R^N$ denotes the vector with $(x^2)_j = (x_j)^2$ for all $j\in\intset{N}$.

\citet{benTal2014hidden} and \citet{locatelli2016exactness} study  the following SOCP relaxation
\begin{align}
\label{eq:socp}
\Opt_{\textup{SOCP}}\coloneqq \inf_{x\in\R^N,\,y\in\R^N} \set{\ip{a_0, y} + 2\ip{b_0, x} + c_0 :\, \begin{array}
	{l}
	\ip{a_i, y} + 2\ip{b_i, x} + c_i \leq 0,\\
	\hspace{8em}\forall i\in\intset{m_I}\\
	\ip{a_i, y} + 2\ip{b_i, x} + c_i  = 0,\\
	\hspace{5em}\forall i\in\intset{m_I+1,m}\\
	y_j\geq x_j^2,\,\forall j\in\intset{N}
\end{array}}.
\end{align}
Let $\mc D_\textup{SOCP}$ denote the epigraph of \eqref{eq:socp} projected onto the $(x,t)$ variables, i.e., define
\begin{align}
\label{eq:socp_epi}
\mc D_\textup{SOCP} \coloneqq \set{(x,t)\in\R^{N+1} :\, \begin{array}
	{l}
	\exists y\in\R^N:\\
	\ip{a_0, y} + 2\ip{b_0, x} + c_0 \leq 2t\\
	\ip{a_i, y} + 2\ip{b_i, x} + c_i \leq 0 ,\,\forall i\in\intset{m_I}\\
	\ip{a_i, y} + 2\ip{b_i, x} + c_i  = 0 ,\,\forall i\in\intset{m_I+1,m}\\
	y_j\geq x_j^2,\,\forall j\in\intset{N}
\end{array}}.
\end{align}

The following proposition states that the SDP and the SOCP relaxations are equivalent for both the convex hull question and the tightness question. In particular, we may apply the sufficient conditions of this paper for either result directly to the SOCP relaxation as well.
\begin{restatable}
	{proposition}{socpsdpequivalence}
	\label{prop:SOCP_SDP_equivalence}
For any SD-QCQP, we have
\begin{align*}
\mc D_\textup{SOCP} = \mc D_\textup{SDP} \quad\text{and}\quad \Opt_\textup{SOCP} = \Opt_\textup{SDP}.
\end{align*}
\end{restatable}

The second identity in Proposition~\ref{prop:SOCP_SDP_equivalence} was first recorded by \citet{locatelli2016exactness}. The first identity, while straightforward, is to the best of our knowledge not present in the literature prior to our work. The proof of this result is deferred to Appendix~\ref{ap:proof_prop_socp_sdp}.

\begin{remark}
	\label{rem:socp-contrast}
	Remark~\ref{rem:soc_representability} implies that for any SD-QCQP satisfying Assumption~\ref{as:gamma_definite}, the set $\mc D_\textup{SDP}$ is SOC representable in the original space. However, this representation may potentially involve exponentially many quadratics---this follows as $\Gamma$ may have exponentially many extreme points and rays. Moreover, identifying these extreme points and rays may require non-trivial computational effort. In contrast, Proposition~\ref{prop:SOCP_SDP_equivalence} implies that $\cD_\textup{SDP}=\cD_\textup{SOCP}$ is SOCP representable in a lifted space (with only $N$ new variables) using only linearly many convex quadratic constraints.
	Consequently, the $\mc D_\textup{SOCP}$ representation is perhaps more interesting from a computational view.\mathprog{\qed}
\end{remark}

 \subsection{On the sharpness of Theorems~\ref{thm:conv_hull_main}~and~\ref{thm:conv_hull_symmetries}}
\label{subsec:sharpness}

In this section we construct QCQPs that show that the assumptions made in Theorem~\ref{thm:conv_hull_symmetries} (and hence in Theorem~\ref{thm:conv_hull_main}) cannot be weakened individually.

We first examine the quadratic eigenvalue multiplicity assumption in Theorems~\ref{thm:conv_hull_main}~and~\ref{thm:conv_hull_symmetries}, and show that both of these theorems break when the assumption on the lower bound on the value of the quadratic eigenvalue multiplicity $k$,
\begin{align*}
k\geq \aff\dim(\set{b(\gamma):\, \gamma\in\mc F})+1
\end{align*}
is relaxed to $k\geq \aff\dim(\set{b(\gamma):\, \gamma\in\mc F})$.
 
\begin{proposition}
\label{prop:aff_dim_sharp_example}
For any positive integers $n$ and $k$, there exists a QCQP in $N\coloneqq nk$ variables with $m\coloneqq k+1$ constraints such that
\begin{itemize}
	\item Assumptions~\ref{as:gamma_definite}~and~\ref{as:gamma_polyhedral} are satisfied,
	\item the quadratic eigenvalue multiplicity of the QCQP is $k$, and
	\item $k$ satisfies
\begin{align*}
k\geq \aff\dim(\set{b(\gamma):\, \gamma\in\mc F})
\end{align*}
for all semidefinite faces $\mc F$ of $\Gamma$, but
\item $\Opt \neq \Opt_\textup{SDP}$ (and hence $\conv(\mc D) \neq \mc D_\textup{SDP}$).
\end{itemize}
\end{proposition}
\begin{proof}
Consider the following QCQP
\begin{align}
	\label{eq:polyhedral_lower_bound}
	\min_{x\in\R^N} \set{- x_1^2 -x_{n+1}^2 - \dots - x_{(k-1)n + 1}^2 :\, \begin{array}
		{l}
		\norm{x}^2 - 1 \leq 0\\
		x_{(j-1)n + 1} = 0 ,\,\forall j\in\intset{1,k}
	\end{array}}.
\end{align}
Here, $A_0 = I_k \otimes \left(-e_1e_1^\top\right)$, $A_1 = I$, and $A_i = 0$ for all $i \in\intset{2,m}$.

Assumption~\ref{as:gamma_definite} holds because $A_1 = I\succ 0$ and $x = 0$ is feasible in \eqref{eq:polyhedral_lower_bound}. Moreover,  Assumption~\ref{as:gamma_polyhedral} holds because 
\begin{align*}
\Gamma \coloneqq \set{\gamma\in\R^m:\, \gamma_1\geq 0 ,\, A(\gamma)\succeq 0 } = \set{\gamma\in\R^m:\, \gamma_1\geq 1}.
\end{align*}
We compute: $\aff\dim(\set{b(\gamma):\, \gamma_1 = 1}) = k$.

By Lemma~\ref{lemma:sdp_in_terms_of_Gamma},
\begin{align*}
\Opt_\textup{SDP} &= \min_{x\in\R^N} \sup_{\gamma\in\Gamma} q(\gamma,x)\leq \sup_{\gamma\in\Gamma} q(\gamma,0)= -1.
\end{align*}
On the other hand, it is clear from \eqref{eq:polyhedral_lower_bound} that $\Opt = 0$.
\end{proof}

We next provide a construction that illustrates that Theorems~\ref{thm:conv_hull_main}~and~\ref{thm:conv_hull_symmetries} both break when Assumption~\ref{as:gamma_polyhedral} is dropped.
\begin{proposition}
\label{prop:Gamma_sharp_example} 
There exists a QCQP in $n = 2$ variables with $m=2$ constraints such that
\begin{itemize}
	\item Assumptions~\ref{as:gamma_definite}~and~\ref{as:cF_well_defined} are satisfied,
	\item the quadratic eigenvalue multiplicity of the QCQP is $k = 1$, and
	\item $k$ satisfies
	\begin{align*}
	k \geq \aff\dim(\set{b(\gamma):\,\gamma\in\mc F}) + 1
	\end{align*}
	for all semidefinite faces $\mc F$ of $\Gamma$, but
	\item $\Opt\neq \Opt_\textup{SDP}$ (and hence $\conv(\mc D) \neq \mc D_\textup{SDP}$).
\end{itemize}
\end{proposition}
\begin{proof}
Consider the following QCQP
\begin{align}
\label{eq:circle_gamma_example}
\min_{x\in\R^2}\set{\norm{x-e_1}^2 :\, \begin{array}
	{l}
	x_1^2 -x_2^2 + 2x_1x_2 = 0\\
	x_1^2 - x_2^2 - 2x_1x_2 = 0
\end{array}}.
\end{align}
Here
\begin{align*}
A_0 = \begin{pmatrix}
	1 & 0 \\ 0 & 1
\end{pmatrix}
,\qquad
A_1 = \begin{pmatrix}
	1 & 1 \\ 1 & -1
\end{pmatrix}
,\qquad
A_2 = \begin{pmatrix}
	1 & -1 \\ -1 & -1
\end{pmatrix}.
\end{align*}

Assumption~\ref{as:gamma_definite} holds since $A(0) = I\succ 0$ and $x=0$ is feasible in \eqref{eq:circle_gamma_example}.

We will describe $\Gamma$ explicitly.
For a $2\by2$ matrix $A(\gamma)$, we have that $A(\gamma)\succeq 0$ if and only if $\tr(A(\gamma))\geq 0$ and $\det(A(\gamma))\geq 0$. Note that $\tr(A(\gamma))=\tr(A_0)\geq 0$ for all $\gamma$, thus
\begin{align*}
\Gamma &= \set{\gamma\in\R^2:\, (1+\gamma_1+\gamma_2)(1-\gamma_1-\gamma_2) - (\gamma_1-\gamma_2)^2\geq 0}\\
&= \set{\gamma\in\R^2:\, 1 - 2\norm{\gamma}^2 \geq 0}\\
&= \mb B(0, 2^{-1/2}).
\end{align*}
Then Assumption~\ref{as:cF_well_defined} holds as $\Gamma$ is bounded.

It is clear that $k\geq 1$. To see that $k=1$, note that $A_1$ has eigenvalues $1$ and $-1$. Furthermore, as $b_1=b_2 = 0$, we have that $\aff\dim(\set{b(\gamma):\gamma\in\R^2}) + 1 = 1$. In particular, the same is true for any semidefinite face $\mc F$ of $\Gamma$.

Next we compute $\Opt_\textup{SDP}$. 
By our explicit description of $\Gamma$, for any fixed $\hat x$ we have
\begin{align*}
\sup_{\gamma\in\Gamma} q(\gamma,\hat x) &= q_0(x) + \max_{\gamma\in\mb B(0,1/\sqrt{2})} \ip{\gamma,\begin{pmatrix}
	q_1(x)\\
	q_2(x)
\end{pmatrix}}\\
&= q_0(x) + \sqrt{(q_1(x)^2 + q_2(x)^2)/2}\\
&= q_0(x) + \norm{x}^2 .
\end{align*}
Then, by Lemma~\ref{lemma:sdp_in_terms_of_Gamma}
\begin{align*}
\Opt_\textup{SDP} &= \min_x \sup_{\gamma\in\Gamma} q(\gamma, x)\\
&= \min_x\left(\norm{x-e_1}^2 + \norm{x}^2\right)\\
&= 1/2.
\end{align*}
On the other hand, it is clear from \eqref{eq:circle_gamma_example} that $\Opt = 1$.
\end{proof}

\section{Exactness of the SDP relaxation}
\label{sec:tightness}

In this section, we use our framework to give new conditions under which $\Opt_\textup{SDP} = \Opt$.

\begin{theorem}
\label{thm:sdp_tightness_main}
Suppose Assumptions~\ref{as:gamma_definite}~and~\ref{as:gamma_polyhedral} hold. If for every semidefinite face $\mc F$ of $\Gamma$ we have 
\begin{align*}
0\notin \Proj_{\mc V(\mc F)}\set{b(\gamma):\, \gamma\in\mc F}, 
\end{align*}
then any optimizer $(x^*,t^*)$ in $\argmin_{(x,t)\in\mc D_{\textup{SDP}}} 2t$ satisfies $(x^*,t^*)\in\mc D$. In particular, $\Opt=\Opt_\textup{SDP}$.
\end{theorem}

In other words, under the assumptions of Theorem~\ref{thm:sdp_tightness_main}, given any optimizer
\begin{align*}
\begin{pmatrix}
	1 & x^\top\\ x & X
\end{pmatrix}
\end{align*}
of \eqref{eq:shor_sdp}, we can simply return $x$ as an optimizer for \eqref{eq:qcqp}.

\begin{proof}
Let
\begin{align*}
(x^*,t^*)\in\argmin_{(x,t)\in\mc D_\textup{SDP}} 2t.
\end{align*}

Let $\mc F = \mc F(x^*)$. We claim that $\mc F$ will always be definite under the assumptions of this theorem. In particular, we will be able to apply Lemma~\ref{lem:F_contains_definite} to conclude that $(x^*,t^*)\in\mc D$. To this end, we will show that $\mc F$ is definite by first assuming that $\mc F$ is semidefinite and then deriving a contradiction to the assumption that $(x^*,t^*)\in\argmin_{(x,t)\in\mc D_\textup{SDP}} 2t$.

Assume for contradiction that $\mc F$ is a semidefinite face of $\Gamma$. By Lemma~\ref{lem:semidefinite_shares_zero}, $\mc V(\mc F)$ has a nonzero element.  For the sake of convenience, let $\mc P\coloneqq \Proj_{\mc V(\mc F)}\set{b(\gamma):\, \gamma\in\mc F}$. 
Assumption~\ref{as:gamma_polyhedral} implies that $\mc P$ is a nonempty closed convex set. Indeed, $\mc P$ is an affine transformation of $\mc F$, which is a face of the polyhedral set $\Gamma$, and is thus itself polyhedral.

Under our assumption, the compact set $\set{0}$ and the nonempty closed convex set $\mc P$ are disjoint. Thus, by the hyperplane separation theorem, there exists a nonzero vector $v\in\mc V(\mc F)$ and $\epsilon>0$ such that $v^\top b(\gamma)\leq -\epsilon$ for all $\gamma\in\mc F$.

Apply Lemma~\ref{lem:decompose_gamma} to decompose $\Gamma = \Gamma_e + \cone(\Gamma_r)$ and $\mc F = \mc F_e + \mc F_r$.

We will modify $(x^*,t^*)$ in the $(v,-\epsilon)$ direction. Define
\begin{align*}
(x_\alpha,t_\alpha) \coloneqq (x^* + \alpha v,t^* - \alpha \epsilon),
\end{align*}
where $\alpha>0$ will be chosen later.

First, consider how $q(\gamma, x_\alpha)-2t_\alpha$ changes with $\alpha$ for fixed $\gamma_f\in\mc F$. We can expand
\begin{align*}
q(\gamma_f, x_\alpha) -2t_\alpha &= \left(q(\gamma_f, x^*) - 2t^*\right) + 2\alpha\left(x^{*\top}A(\gamma_f)v + b(\gamma_f)^\top v + \epsilon \right) + \alpha^2 v^\top A(\gamma_f)v\\
&\leq \left(q(\gamma_f, x^*)- 2t^*\right)\\
&= 0.
\end{align*}
The second line follows as $A(\gamma_f)v = 0$ and $b(\gamma_f)^\top v \leq -\epsilon$ for all $\gamma_f\in\mc F$.
The third line follows as $q(\gamma_f,x^*) = 2t^*$ for all $\gamma_f\in\mc F$.

On the other hand, for $\gamma_e\in\Gamma_e\setminus \mc F_e$, the function $\alpha\mapsto q(\gamma_e, x_\alpha)-2t_\alpha$ is a continuous function taking the value $q(\gamma_e, x^*)-2t^*<0$ at $\alpha = 0$ (the strict inequality follows from the fact that $\gamma_e\in\Gamma_e\setminus \mc F_e$).

Similarly, for $\gamma_r\in\Gamma_r\setminus \mc F_r$, the function $\alpha\mapsto \breve q(\gamma_r, x_\alpha)$ is a continuous function taking the value $\breve q(\gamma_r, x^*)<0$ at $\alpha =0$ (the strict inequality follows from the fact that $\gamma_r\in\Gamma_r\setminus \mc F_r$).

We have shown that the following finite set of continuous functions in $\alpha$,
\begin{align*}
\mc Q \coloneqq \set{q(\gamma_e,x_\alpha)-2t_\alpha :\, \gamma_e\in\extr(\Gamma_e)\setminus \mc F_e} \cup \set{\breve q(\gamma_r, x_\alpha) :\, \gamma_r \in\extr(\Gamma_r)\setminus \mc F_r},
\end{align*}
consists of continuous functions which are negative at $\alpha = 0$. The finiteness of this set follows from the assumption that $\Gamma$ is polyhedral.

Fix an $\alpha>0$ such that $q(\alpha)\leq 0$ for every $q\in\mc Q$ --- this is possible by the finiteness of $\mc Q$ and the continuity of each $q\in\mc Q$.
Then because $\Gamma_e = \conv(\extr(\Gamma_e))$ and $\Gamma_r = \conv(\extr(\Gamma_r))$, we have
$q(\gamma_e,x_\alpha)\leq 2t_\alpha$ for all $\gamma_e\in\Gamma_e$ and $\breve q(\gamma_r,x_\alpha)\leq 0$ for all $\gamma_r\in\Gamma_r$. Thus, $(x_\alpha, t_\alpha)\in\mc D_\textup{SDP}$.
In particular, $\min_{(x,t)\in\mc D_\textup{SDP}}2t \leq2 t_\alpha <2t^*$, a contradiction.
\end{proof}

The following theorem will follow from Theorem~\ref{thm:sdp_tightness_main} by a perturbation argument.
\begin{theorem}
\label{thm:sdp_tightness_perturbed}
Suppose Assumptions~\ref{as:gamma_definite}~and~\ref{as:gamma_polyhedral} hold. If there exists a sequence $(h_j)_{j\in\N}$ in $\R^N$ such that $\lim_{j\to\infty}h_j=0$ and for every semidefinite face $\mc F$ of $\Gamma$ and $j\in\N$ we have
\begin{align*}
0 \notin \Proj_{\mc V(\mc F)}\set{b(\gamma) + h_j:\, \gamma\in\mc F}, 
\end{align*}
then $\Opt = \Opt_\textup{SDP}$.
\end{theorem}

\begin{proof}
Consider the following sequence of QCQPs indexed by $j\in\N$:
\begin{align*}
\Opt_j \coloneqq \min_{x\in\R^N} \set{q_0(x) + 2h_j^\top x :\, \begin{array}
	{l}
	q_i(x) \leq 0 ,\,\forall i\in\intset{m_I}\\
	q_i(x) = 0 ,\,\forall i\in\intset{m_I+1,m}
\end{array}}.
\end{align*}
We will use the subscript $j$ to denote all quantities corresponding to the perturbed QCQP. 
By construction, each of the QCQPs in this sequence satisfies the assumptions of Theorem~\ref{thm:sdp_tightness_main} and thus $\Opt_{\textup{SDP},j} = \Opt_j$. 
For $j\in\N$, let
\begin{align*}
(x_j, t_j)\in\argmin_{(x,t)\in\mc D_j} 2t.
\end{align*}
Let $x^*$ be a subsequential limit of $\set{x_j}_{j\in\N}$ (this exists as we can bound the sequence $\set{x_j}_{j\in\N}$ using Assumption~\ref{as:gamma_definite}). Noting that the feasible domain of the original QCQP is closed, we have that $x^*$, a subsequential limit of feasible points, is also feasible. Finally, by continuity of $q_0$ and the optimality of $(x_j,t_j)\in\mc D_j$, we have that
\begin{align*}
q_0(x^*) = \lim_{j\to\infty} q_0(x_j) = \lim_{j\to\infty} \Opt_j = \lim_{j\to\infty} \Opt_{\textup{SDP},j} = \Opt_\textup{SDP}.
\end{align*}
Here, the final equality holds by a simple boundedness argument and Assumption~\ref{as:gamma_definite}.
\end{proof}

The following example shows that SDP tightness (for example via Theorem~\ref{thm:sdp_tightness_perturbed}) may hold even when the convex hull result does not.
\begin{example}
\label{ex:optimality}
Consider the following QCQP
\begin{align*}
\inf_{x\in\R^2} \set{x_1^2 + x_2^2 :\, \begin{array}
	{l}
	x_1^2  - x_2^2 \leq 0 \\
	2x_2 \leq 0
\end{array}}.
\end{align*}
We verify that the conditions of Theorem~\ref{thm:sdp_tightness_perturbed} hold. It is clear that Assumption~\ref{as:gamma_definite} holds: $A(0) = I \succ 0$ and $x=0$ is feasible. It is easy to verify that $\Gamma = [0,1] \times\R_+$, thus Assumption~\ref{as:gamma_polyhedral} also holds.
Finally, pick $h_j = e_2/j$ for $j\in\N$. Note that the only semidefinite face of $\Gamma$ is $\mc F = \set{1} \times \R_+$ and that $\mc V(\mc F) = \spann\set{e_2}$. In particular,
\begin{align*}
\Proj_{\mc V(\mc F)}\set{b(\gamma) + h_j:\, \gamma\in\mc F} = \set{0}\times [1/j,\infty),
\end{align*}
which does not contain $0$. We deduce that $\Opt = \Opt_\textup{SDP}$.

Next, we claim that $\conv(\mc D) \neq \mc D_\textup{SDP}$. First note that $\mc D$ is actually convex in this example.
\begin{align*}
\mc D &= \set{(x,t):\, \begin{array}
	{l}
	x_1^2 + x_2^2 \leq 2t\\
	x_1^2 - x_2^2 \leq 0\\
	2x_2 \leq 0
\end{array}}
=\set{(x,t):\, \begin{array}
	{l}
	x_1^2 + x_2^2 \leq 2t\\
	\abs{x_1}  \leq  -x_2\\
	2x_2 \leq 0
\end{array}}
\end{align*}
Next by Lemma~\ref{lemma:sdp_in_terms_of_Gamma} and the description of $\Gamma$ above, we have that
\begin{align*}
\mc D_\textup{SDP} &= \set{(x,t):\, \begin{array}
	{l}
	x_1^2+x_2^2 \leq 2t\\
	2x_1^2 \leq 2t\\
	2x_2\leq 0
\end{array}}.
\end{align*}
Then we may check, for example, that
\begin{align*}
((1,0),1)\in\mc D_\textup{SDP} \qquad\text{but}\qquad ((1,0),1)\notin\mc D = \conv(\mc D).
\end{align*}
We conclude that $\Opt=\Opt_\textup{SDP}$ but $\conv(\mc D) \neq \mc D_\textup{SDP}$. We plot $\mc D$ and $\mc D_\textup{SDP}$ in Figure~\ref{fig:example_optimality}.\mathprog{\qed}
\end{example}
\begin{figure}
  \centering
    \includegraphics[width=0.4\textwidth]{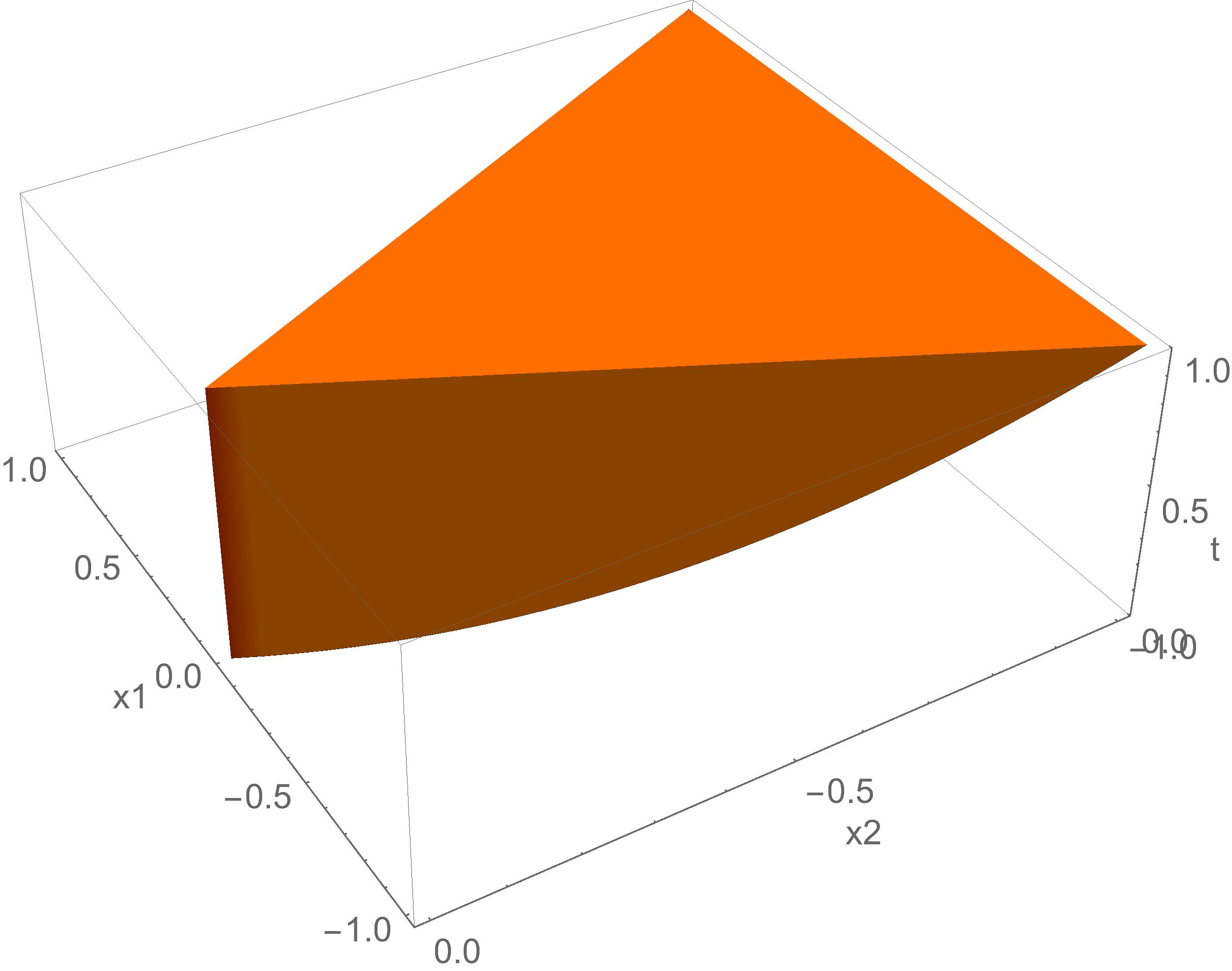}\qquad
    \includegraphics[width=0.4\textwidth]{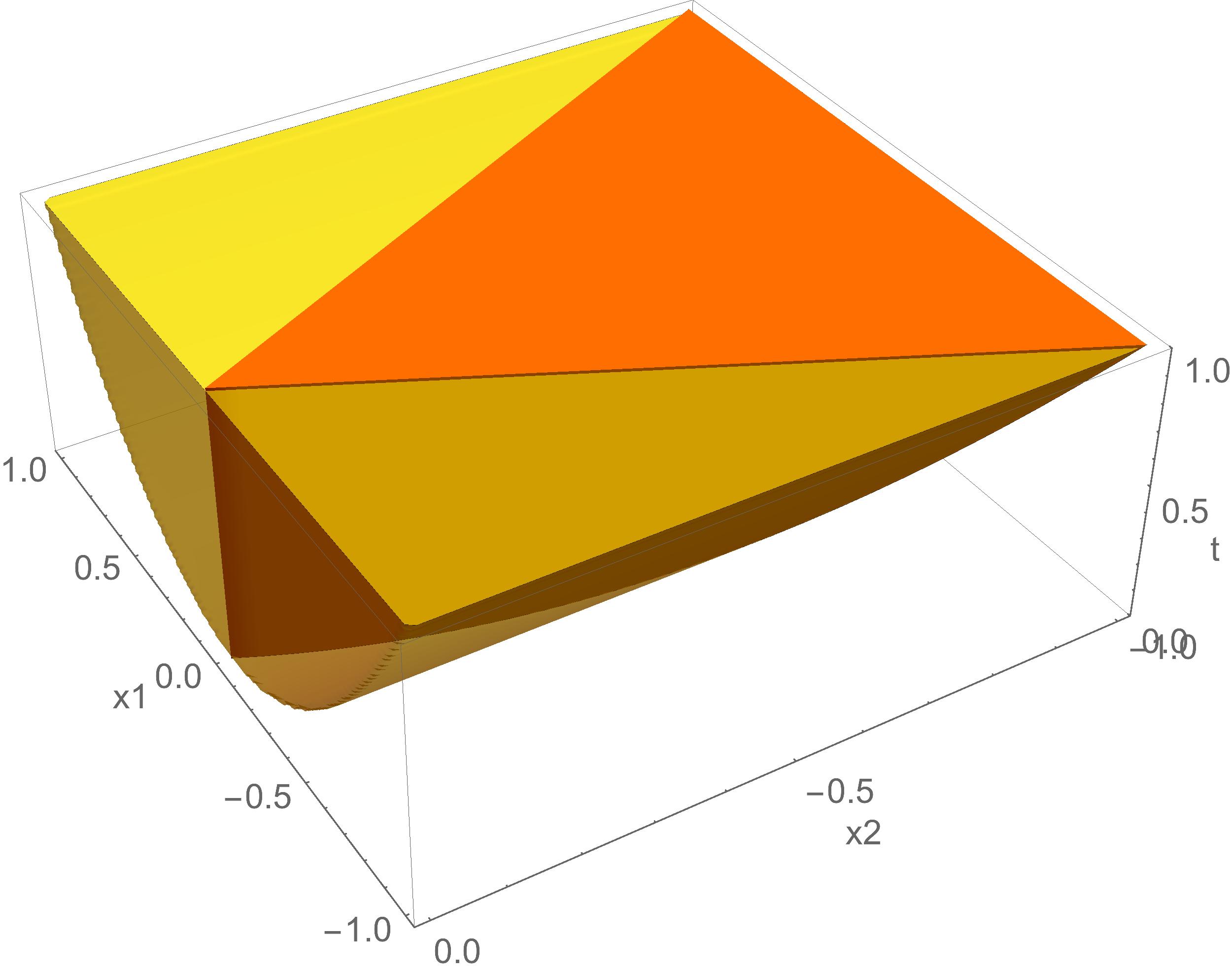}
	\caption{The sets $\conv(\mc D)$ (in orange) and $\mc D_\textup{SDP}$ (in yellow) from Example~\ref{ex:optimality}}
	\label{fig:example_optimality}
\end{figure}

\subsection{Comparison with related conditions in the literature}
\label{subsec:literature_comparison}

Several sufficient conditions for SDP tightness results have been examined in the literature. In this section, we compare these conditions with our Theorems~\ref{thm:sdp_tightness_main}~and~\ref{thm:sdp_tightness_perturbed}.

\citet{locatelli2016exactness} considers the SDP relaxation of a variant of the TRS,
\begin{align}
\label{eq:locatelli}
\inf_{x\in\R^N} \set{q_0(x) :\, \begin{array}
	{l}
	b_i^\top x + c_i \leq 0 ,\,\forall i\in\intset{m - 1}\\
	x^\top x -1 \leq 0
\end{array}}.
\end{align}
We assume that $A_0 = \Diag(a_0)$ without loss of generality. Indeed, if $A_0$ is not diagonal, we can reformulate the problem in the eigenbasis of $A_0$.
Furthermore, we will assume that $A_0$ has at least one negative eigenvalue as otherwise \eqref{eq:locatelli} is already convex.

Let $J\subseteq \intset{N}$ be the set of coordinates corresponding to $\lambda_{\min}(A_0)$, i.e., define
\begin{align*}
J \coloneqq \set{j\in\intset{N}: (a_0)_j = \min_{i\in\intset{N}} (a_0)_i},
\end{align*}
and let $\mc V_J \coloneqq \spann(\set{e_j: j\in J})$.

\citet{locatelli2016exactness} derives a sufficient condition for SDP tightness by reasoning about the nonexistence of certain KKT multipliers in the SOCP relaxation of \eqref{eq:locatelli}. For the sake of completeness, we restate this result in our language.

\begin{theorem}
[{\cite[Theorem 3.1]{locatelli2016exactness}}]
\label{thm:locatelli}
Consider the problem \eqref{eq:locatelli} and assume that $A_0$ has at least one negative eigenvalue. Suppose the feasible region of \eqref{eq:locatelli} is strictly feasible. If there exists a sequence $(h_j)_{j\in\N}$ in $\R^N$ such that $\lim_{j\to\infty} h_j = 0$ and for every $j\in\N$ we have
\begin{align*}
0\notin \Proj_{\mc V_J}\set{b(\gamma)+ h_j:\, 
\gamma\in\R_+^m},
\end{align*}
then $\Opt = \Opt_\textup{SDP}$.
\end{theorem}

\begin{proposition}\label{prop:locatelli_connection}
Suppose the assumptions of Theorem~\ref{thm:locatelli} hold, then the assumptions of Theorem~\ref{thm:sdp_tightness_perturbed} also hold.
\end{proposition}
\begin{proof}
Consider a QCQP of the form \eqref{eq:locatelli} satisfying the assumptions of Theorem~\ref{thm:locatelli}. We will verify that the assumptions of Theorem~\ref{thm:sdp_tightness_perturbed} are also satisfied. Note the feasible region of \eqref{eq:qcqp} is nonempty. Furthermore, by taking $\eta\in\R$ large enough, we can ensure that $A(\eta e_m) = A_0 + \eta I\succ 0$. Thus, Assumption~\ref{as:gamma_definite} is satisfied. Assumption~\ref{as:gamma_polyhedral} is satisfied as well because 
\begin{align}
\label{eq:locatelli_gamma}
\Gamma = \set{\gamma\in\R^m:\, \begin{array}
	{l}
	A(\gamma)\succeq 0\\
	\gamma\geq 0
\end{array}} = \set{\gamma\in\R^m:\, \begin{array}
	{l}
	\gamma_m \geq -\lambda_{\min}(A_0)\\
	\gamma \geq 0
\end{array}}
\end{align}
is polyhedral.

Let $\mc F$ be a semidefinite face of $\Gamma$.
By Lemma~\ref{lem:semidefinite_shares_zero}, $A(\gamma)$ must have a zero eigenvalue for every $\gamma\in\mc F$. In particular, we can deduce from the description of $\Gamma$ in \eqref{eq:locatelli_gamma} that
\begin{align*}
\mc F = \set{\gamma\in\R^m:\, \begin{array}
	{l}
	\gamma_m = -\lambda_{\min}(A_0)\\
	\gamma \geq 0
\end{array}}.
\end{align*}
Therefore, $\mc V(\mc F) = \mc V_J$. Then the assumption $0\notin \Proj_{\mc V_J}\set{b(\gamma)+h_j:\, \gamma\in\R_+^m}$ for every $j\in\N$ immediately implies that
\begin{align*}
0\notin \Proj_{\mc V(\mc F)}\set{b(\gamma)+h_j:\, \gamma\in\mc F}
\end{align*}
for every $j\in\N$ as $\R_+^m \supseteq \mc F$. Hence, we conclude that the third condition in Theorem~\ref{thm:sdp_tightness_perturbed} also holds.
\end{proof}

\begin{remark}
\citet{hoNguyen2017second} study a particular convex relaxation of the TRS with additional conic constraints. For such problems, they suggest a particular assumption under which their relaxation is tight; see \cite[Theorem 2.4]{hoNguyen2017second}. It was also shown in \cite[Lemma 2.10]{hoNguyen2017second} that when the conic constraints are in a particular linear form, then their assumption is indeed an equivalent form of \citet{locatelli2016exactness}'s assumption from Theorem~\ref{thm:locatelli}. 
It is of interest to compare our assumptions with the one from \cite{hoNguyen2017second}. 
We note however that our Theorem~\ref{thm:sdp_tightness_perturbed} and the result due to \cite[Theorem 2.4]{hoNguyen2017second} are incomparable. To see this, note that the former covers some optimization problems with nonconvex quadratic constraints while the latter covers some optimization problems with non-quadratic conic constraints. In addition, we note that the relaxation studied in \citet{hoNguyen2017second} is weaker than the SDP relaxation that we study here.\mathprog{\qed}
\end{remark}

\citet{burer2019exact} consider the standard SDP relaxation of diagonal QCQPs\footnote{\citet{burer2019exact} address general QCQPs in their paper by first transforming them into diagonal QCQPs and then applying the standard SDP relaxation. In particular, the standard Shor SDP relaxation is only analyzed in the context of diagonal QCQPs.} and show that under an assumption on the input data $\set{A_i}_{i\in\intset{0,m}}$ and $\set{b_i}_{i\in\intset{0,m}}$ that the SDP relaxation is tight.
For the sake of completeness, we first restate\footnote{The original statement of this theorem gives additional guarantees, which are weaker than SDP tightness, when the conditions of Theorem~\ref{thm:burer_ye} fail.} \cite[Theorem 1]{burer2019exact} as it relates to SDP tightness in our language.

\begin{theorem}
[{\cite[Theorem 1]{burer2019exact}}]
\label{thm:burer_ye}
Consider a diagonal QCQP with no equality constraints.
Suppose the feasible region of \eqref{eq:qcqp} is nonempty and there exists $\gamma^*\geq 0$ such that $\breve A(\gamma^*)\succ 0$. Suppose the SDP relaxation \eqref{eq:shor_sdp} is strictly feasible. If 
for every $j\in\intset{N}$
 the set
\begin{align*}
\set{\gamma\in\R^m :\, \begin{array}
	{l}
	\gamma\geq 0\\
	A(\gamma)\succeq 0\\
	A(\gamma)_{j,j} = 0\\
	b(\gamma)_j = 0
\end{array}}
\end{align*}
is empty, 
then any optimizer $(x^*, t^*)$ in $\argmin_{(x,t)\in\mc D_\textup{SDP}} 2t$ satisfies $(x^*,t^*)\in\mc D$.
\end{theorem}

\begin{proposition}\label{prop:BYconnection}
Suppose the assumptions of Theorem~\ref{thm:burer_ye} hold, then the assumptions of Theorem~\ref{thm:sdp_tightness_main} also hold.
\end{proposition}
\begin{proof}
Consider a QCQP satisfying the assumptions of Theorem~\ref{thm:burer_ye}. We will verify that the assumptions of Theorem~\ref{thm:sdp_tightness_main} are also satisfied.
Note the feasible region of \eqref{eq:qcqp} is nonempty. Furthermore, by taking $\eta\in\R$ large enough, we can ensure $A(\eta\gamma^*) = A_0 + \eta \breve A(\gamma^*)\succ 0$. Thus, Assumption~\ref{as:gamma_definite} is satisfied.
Assumption~\ref{as:gamma_polyhedral} holds as all of the quadratic forms $A_0$, $\dots$, $A_m$ are diagonal.
The condition on the input data in Theorem~\ref{thm:burer_ye} is equivalent to requiring that
\begin{align*}
A(\gamma)_{j,j} = 0 \implies b(\gamma)_j \neq 0
\end{align*}
for all $\gamma\in\Gamma$ and $j\in\intset{N}$. Consider a semidefinite face $\mc F$ of $\Gamma$, and any $\gamma\in\mc F$. As $A(\gamma)$ is diagonal, we deduce  that
\begin{align*}
\mc V(\mc F) = \spann(\set{e_j :\, A(\gamma)_{j,j} =0}).
\end{align*}
Then, the final assumption in Theorem~\ref{thm:sdp_tightness_main}, namely
\begin{align*}
0\notin\Proj_{\mc V(\mc F)}\set{b(\gamma):\, \gamma\in\mc F},
\end{align*}
holds immediately.
\end{proof}

The following example shows that Theorem~\ref{thm:sdp_tightness_main} is strictly more general than Theorem~\ref{thm:burer_ye} even in the case of diagonal QCQPs with strictly convex constraints.

\begin{example}
\label{ex:separating_BY}
Consider the following QCQP
\begin{align*}
\min_{x\in\R^2}\set{-\norm{x}^2 :\, \begin{array}
	{l}
	\norm{x - e_1}^2 \leq 1\\
	\norm{x-e_2}^2 \leq 1
\end{array}}.
\end{align*}
We first verify that the assumptions of Theorem~\ref{thm:sdp_tightness_main} hold. It is clear that this problem satisfies Assumption~\ref{as:gamma_definite}: the origin is feasible and $A(e_1+e_2) = I\succ 0$. Next, we compute $\Gamma$.
\begin{align*}
\Gamma = \set{\gamma\in\R^2_+:\, A(\gamma)\succeq 0} &= \set{\gamma\in\R^2_+:\, \gamma_1+\gamma_2\geq 1}.
\end{align*}
We conclude that Assumption~\ref{as:gamma_polyhedral} also holds. Furthermore, the only semidefinite face of $\Gamma$ is $\mc F = \set{\gamma\in\R^2_+:\, \gamma_1+\gamma_2 = 1}$. For this semidefinite face, we have that $\mc V(\mc F)$ is the entire space $\R^2$. Consequently,
\begin{align*}
\Proj_{\mc V(\mc F)} \set{b(\gamma):\, \gamma\in\mc F} = \set{\gamma_1 e_1 + \gamma_2e_2 :\, \gamma\in\R^2_+,\, \gamma_1+\gamma_2 = 1}
\end{align*}
is the set of all convex combinations of $e_1$ and $e_2$. This set does not contain the origin and thus the assumptions of Theorem~\ref{thm:sdp_tightness_main} are satisfied.

On the other hand, by picking $j = 1$ in Theorem~\ref{thm:burer_ye} and $\gamma = e_2$, we have that $\gamma\geq 0$, $A(\gamma) \succeq 0$, and $A(\gamma)_{j,j} = 0$ but $b(\gamma)_j = (e_2)_1 = 0$. We see that the assumptions of Theorem~\ref{thm:burer_ye} are not satisfied.\mathprog{\qed}
\end{example} \section{Removing the polyhedrality assumption}
\label{sec:removing_polyhedral}

One of the main assumptions we use in our proof of the convex hull results (Theorems~\ref{thm:conv_hull_main}~and~\ref{thm:conv_hull_symmetries}) and the SDP tightness results (Theorems~\ref{thm:sdp_tightness_main}~and~\ref{thm:sdp_tightness_perturbed}) is that the set $\Gamma$ is polyhedral (Assumption~\ref{as:gamma_polyhedral}).
In this section we show that one can remove Assumption~\ref{as:gamma_polyhedral} in Theorem~\ref{thm:conv_hull_symmetries} when $k$ is sufficiently large\footnote{
Recall the example constructed in Proposition~\ref{prop:Gamma_sharp_example}. This example shows that both the convex hull result and SDP tightness result fail when Assumption~\ref{as:gamma_polyhedral} is dropped from Theorem~\ref{thm:conv_hull_symmetries}. In particular, the SDP tightness and convex hull results we recover in this section will require assumptions on $k$ that are strictly stronger than in the polyhedral case.
}.
The results in this section do not use the framework described in Section~\ref{sec:framework} and in particular do not require the technical assumption (Assumption~\ref{as:cF_well_defined}).

\begin{theorem}
\label{thm:general_Gamma_conv_hull}
Suppose Assumption~\ref{as:gamma_definite} holds. If the quadratic eigenvalue multiplicity $k$ satisfies $k\geq m+2$, then $\conv(\mc D) = \mc D_\textup{SDP}$.
\end{theorem}
\begin{proof}
Suppose $(\hat x,\hat t)\in\mc D_\textup{SDP}$.
Therefore,
\begin{align*}
2\hat t &\geq \sup_{\gamma\in\R^m} \set{q(\gamma,\hat x) :\, \begin{array}
	{l}
	A(\gamma)\succeq 0\\
	\gamma_i \geq 0 ,\,\forall i\in\intset{m_I}
\end{array}}\\
&= \sup_{\gamma\in\R^m} \set{q(\gamma,\hat x) :\, \begin{array}
	{l}
	\cA(\gamma)\succeq 0\\
	\gamma_i \geq 0 ,\,\forall i\in\intset{m_I}
\end{array}}.
\end{align*}
The second line follows as $A(\gamma)\succeq 0$ if and only if $\cA(\gamma)\succeq 0$.
Note that Assumption~\ref{as:gamma_definite} allows us to apply strong conic duality to the program on the second line. Furthermore, this dual SDP achieves its optimal value, i.e., there exists $Z\in\S^n$ such that $(\hat x,\hat t, Z)$ satisfies
\begin{align}
\label{eq:hatx_hatt_Z_system_conv}
\begin{cases}
q_0(\hat x) + \ip{\cA_0, Z} \leq 2\hat t\\
q_i(\hat x) + \ip{\cA_i, Z} \leq 0,\,\forall i\in\intset{m_I}\\
q_i(\hat x) + \ip{\cA_i, Z} =0,\,\forall i\in\intset{m_I+1,m}\\
Z\succeq 0.
\end{cases}
\end{align}
We will show by induction on $\rank(Z)$ that for any $(\hat x,\hat t, Z)$ satisfying \eqref{eq:hatx_hatt_Z_system_conv}, we have $(\hat x,\hat t)\in\conv(\cD)$. 
The claim clearly holds when $\rank(Z) = 0$.

Now suppose $r\coloneqq \rank(Z)\geq 1$.
Let $(\hat x,\hat t,Z)$ satisfy \eqref{eq:hatx_hatt_Z_system_conv}. Write $Z = \sum_{i=1}^{r}z_iz_i^\top$ where each $z_i$ is nonzero. Fix $z\coloneqq z_1$.

We claim that the following system in $y$ is feasible:
\begin{align}
\label{eq:removing_polyhedral_system_conv}
\begin{cases}
	\ip{A_0\hat x +b_0, y\otimes z} = 0\\
	\ip{A_i\hat x+b_i, y\otimes z} = 0 ,\,\forall i\in\intset{m}\\
	y\in\mb S^{k-1}.
\end{cases}
\end{align}
Indeed, the first two constraints impose at most $m+1$ homogeneous linear equalities in $k\geq m+2$ variables. In particular, there exists a nonzero solution $y$ to the first two constraints. This $y$ may then be scaled to satisfy $y\in\mb S^{k-1}$.

Note then that for all $i \in\intset{0,m}$,
\begin{align*}
q_i(\hat x \pm y\otimes z) &= (\hat x \pm y\otimes z)^\top  A_i (\hat x \pm y\otimes z) + 2 b_i^\top (\hat x \pm y\otimes z) + c_i\\
&= q_i(\hat x) \pm 2\ip{A_i \hat x + b_i, y\otimes z} + \ip{\cA_i, zz^\top}\\
&= q_i(\hat x) + \ip{\cA_i, zz^\top}.
\end{align*}

Consequently, $(\hat x\pm y\otimes z, \hat t, Z - zz^\top)$ satisfies \eqref{eq:hatx_hatt_Z_system_conv}.
Furthermore, we have $\rank(Z-zz^\top)= r-1$. By induction, $(\hat x\pm y\otimes z,\hat t)\in\conv(\cD)$. 
We conclude that $(\hat x,\hat t)\in\conv(\cD)$.
\end{proof}

A similar proof leads to an SDP tightness result without Assumption~\ref{as:gamma_polyhedral}.
\begin{restatable}
	{theorem}{thmgeneralgammatightness}
	\label{thm:general_Gamma_sdp_tightness}
	Suppose Assumption~\ref{as:gamma_definite} holds.
	Define the hyperplane $H = \set{(x,t)\in\R^{N+1}:\, 2t = \Opt_\textup{SDP}}$.
	If the quadratic eigenvalue multiplicity $k$ satisfies $k \geq m+1$, then $\conv(\cD \cap H) = \cD_\textup{SDP} \cap H$. In particular, $\Opt = \Opt_\textup{SDP}$.
\end{restatable}
The proof of this statement follows the proof of Theorem~\ref{thm:general_Gamma_conv_hull} almost exactly and is deferred to Appendix~\ref{ap:proof_thm_general_Gamma_tightness}. For now, we will simply sketch how to modify the proof of Theorem~\ref{thm:general_Gamma_conv_hull} to get a proof for Theorem~\ref{thm:general_Gamma_sdp_tightness}:
We will only consider points $(\hat x, \hat t)\in\cD_\textup{SDP}\cap H$. In this situation, it is easy to show that the first two constraints in \eqref{eq:removing_polyhedral_system_conv} are dependent and impose at most $m$ homogeneous linear equalities. Thus we may carry out the procedure in the proof of Theorem~\ref{thm:general_Gamma_conv_hull} as long as $k\geq m+1$. 
At the end of the procedure, we will have decomposed $(\hat x,\hat t)$ as a convex combination of points $(x_\alpha, \hat t)\in\cD$.

\begin{remark}
\citet[Corollary 4.4]{beck2007quadratic} shows that under Assumption~\ref{as:gamma_definite}, the conclusion $\Opt=\Opt_\textup{SDP}$ holds even when $k = m$.
Thus, recalling the definition of $H$ from Theorem~\ref{thm:general_Gamma_sdp_tightness}, we can summarize Theorems~\ref{thm:general_Gamma_conv_hull}~and~\ref{thm:general_Gamma_sdp_tightness} and~\cite[Corollary 4.4]{beck2007quadratic} as follows. Under Assumption~\ref{as:gamma_definite}, we have:
\begin{center}
\begin{tabular}{@{}lll@{}}\toprule
Assumption & Result & Reference\\\midrule
$k \geq m+2$ & $\conv(\cD) = \cD_\textup{SDP}$ &\quad Theorem~\ref{thm:general_Gamma_conv_hull}\\
$k \geq m+1$ & $\conv(\cD\cap H) = \cD_\textup{SDP}\cap H$ &\quad Theorem~\ref{thm:general_Gamma_sdp_tightness}\\
$k \geq m$ & $\cD \cap H \neq \emptyset$ &\quad \cite[Corollary 4.4]{beck2007quadratic}\\\midrule
\end{tabular}
\end{center}
We conjecture, but are unable to prove at the moment, that the values required of $k$ for these three results are sharp.\mathprog{\qed}
\end{remark} 
\section*{Acknowledgments}
This research is supported in part by NSF grant CMMI 1454548 and ONR grant N00014-19-1-2321.
The authors wish to thank the review team for their feedback and suggestions that led to an improved presentation of the material. 

{
\bibliographystyle{plainnat}

}

\begin{appendix}
\section{Proof of Proposition~\ref{prop:SOCP_SDP_equivalence}}
\label{ap:proof_prop_socp_sdp}

\socpsdpequivalence*
\begin{proof}
The second identity follows immediately from the first identity, thus it suffices to prove only the former.

Let $(x, t)\in\mc D_\textup{SDP}$. By definition, there exists $X\in\Se^N$ such that the following system is satisfied
\begin{align*}
\begin{cases}
	Y\coloneqq \begin{pmatrix}
		1 & x^\top\\ x & X
	\end{pmatrix}\\
	\ip{Q_0, Y} \leq 2t\\
	\ip{Q_i, Y} \leq 0 ,\,\forall i\in\intset{m_I}\\
	\ip{Q_i,Y} = 0 ,\,\forall i\in\intset{m_I+1,m}\\
	Y\succeq 0.
\end{cases}
\end{align*}
Taking a Schur complement of $1$ in the matrix $Y$, we see that $X\succeq xx^\top$. In particular, we have that $X_{j,j} \geq x_j^2$ for all $j\in\intset{N}$. Define the vector $y$ by $y_j = X_{j,j}\geq x_j^2$. Then, noting that $\ip{\Diag(a_i), X} = \ip{a_i, y}$ for all $i\in\intset{0,m}$, we conclude that $(x,t)\in\mc D_\textup{SOCP}$.

Let $(x, t)\in\mc D_\textup{SOCP}$. By definition, there exists $y\in \R^N$ such that the following system is satisfied
\begin{align*}
\begin{cases}
	\ip{a_0, y} + 2\ip{b_0, x} + c_0 \leq 2t\\
	\ip{a_i, y} + 2\ip{b_i, x} + c_i \leq 0 ,\,\forall i\in\intset{m_I}\\
	\ip{a_i, y} + 2\ip{b_i, x} + c_i  = 0 ,\,\forall i\in\intset{m_I+1,m}\\
	y_j\geq x_j^2,\,\forall j\in\intset{N}.
\end{cases}
\end{align*}
Define $X\in\Se^N$ such that $X_{j,j} = y_j$ for all $j\in\intset{N}$ and $X_{j,k} = x_jx_k$ for $j\neq k$. From the definition of $\mc D_\textup{SOCP}$, the relation $y_j \geq x_j^2$ holds for all $j\in\intset{N}$, therefore
\begin{align*}
\begin{pmatrix}
	1 & x^\top \\ x & X
\end{pmatrix}\succeq
\begin{pmatrix}
	1 & x^\top \\
	x & xx^\top 
\end{pmatrix}\succeq 0.
\end{align*}
Finally, noting that $\ip{\Diag(a_i), X} = \ip{a_i, y}$ for all $i\in\intset{0,m}$, we conclude that $(x,t)\in\mc D_\textup{SDP}$.
\end{proof} \section{Proof of Theorem~\ref{thm:general_Gamma_sdp_tightness}}
\label{ap:proof_thm_general_Gamma_tightness}

\thmgeneralgammatightness*
\begin{proof}
Suppose $(\hat x,\hat t)\in\mc D_\textup{SDP}\cap H$.
Then by Lemma~\ref{lemma:sdp_in_terms_of_Gamma} and optimality of $\hat t$, we have that $2\hat t = \sup_{\gamma\in\Gamma}q(\gamma,\hat x)$, i.e.,
\begin{align*}
2\hat t &= \sup_{\gamma\in\R^m} \set{q(\gamma,\hat x) :\, \begin{array}
	{l}
	A(\gamma)\succeq 0\\
	\gamma_i \geq 0 ,\,\forall i\in\intset{m_I}
\end{array}}\\
&= \sup_{\gamma\in\R^m} \set{q(\gamma,\hat x) :\, \begin{array}
	{l}
	\cA(\gamma)\succeq 0\\
	\gamma_i \geq 0 ,\,\forall i\in\intset{m_I}
\end{array}}.
\end{align*}
The second line follows as $A(\gamma)\succeq 0$ if and only if $\cA(\gamma)\succeq 0$.
Note that Assumption~\ref{as:gamma_definite} allows us to apply strong conic duality to the program on the second line. Furthermore, this dual SDP achieves its optimal value, i.e., there exists $Z\in\S^n$ such that $(\hat x,\hat t, Z)$ satisfies
\begin{align}
\label{eq:hatx_hatt_Z_system_obj}
\begin{cases}
q_0(\hat x) + \ip{\cA_0, Z} = 2\hat t\\
q_i(\hat x) + \ip{\cA_i, Z} \leq 0,\,\forall i\in\intset{m_I}\\
q_i(\hat x) + \ip{\cA_i, Z} =0,\,\forall i\in\intset{m_I+1,m}\\
Z\succeq 0.
\end{cases}
\end{align}

We will show by induction on $\rank(Z)$ that for any $(\hat x,\hat t, Z)$ satisfying \eqref{eq:hatx_hatt_Z_system_obj}, we have $(\hat x,\hat t)\in\conv(\cD\cap H)$. 
The claim clearly holds when $\rank(Z) = 0$.

Now suppose $r\coloneqq \rank(Z)\geq 1$.
Let $(\hat x,\hat t,Z)$ satisfy \eqref{eq:hatx_hatt_Z_system_obj}. Write $Z = \sum_{i=1}^{r}z_iz_i^\top$ where each $z_i$ is nonzero. Fix $z\coloneqq z_1$.

We claim that the following system in $y$ is feasible:
\begin{align}
\label{eq:removing_polyhedral_system_obj}
\begin{cases}
	\ip{A_i\hat x+b_i, y\otimes z} = 0 ,\,\forall i\in\intset{m}\\
	y\in\mb S^{k-1}.
\end{cases}
\end{align}
Indeed, the linear constraints impose at most $m$ homogeneous linear equalities in $k\geq m+1$ variables. In particular, there exists a nonzero solution $y$ to the linear constraints. This $y$ may then be scaled to satisfy $y\in\mb S^{k-1}$.

Note then that for all $i \in\intset{1,m}$,
\begin{align*}
q_i(\hat x \pm y\otimes z) + \ip{\cA_i, Z - zz^\top}&= (\hat x \pm y\otimes z)^\top  A_i (\hat x \pm y\otimes z) + 2 b_i^\top (\hat x \pm y\otimes z) + c_i + \ip{\cA_i, Z - zz^\top}\\
&= q_i(\hat x) \pm 2\ip{A_i \hat x + b_i, y\otimes z} + \ip{\cA_i, Z}\\
&= q_i(\hat x) + \ip{\cA_i, Z}.
\end{align*}
Consequently, $(\hat x \pm y \otimes z, \hat t, Z - zz^\top)$ satisfies all of the constraints in \eqref{eq:hatx_hatt_Z_system_obj} except possibly the first.
We now verify that the first constraint is also satisfied: From
\begin{align*}
q_0(\hat x \pm y \otimes z) + \ip{\cA_0, Z - zz^\top} &=  q_0(\hat x) \pm 2\ip{A_0 \hat x + b_0, y\otimes z} + \ip{\cA_0, zz^\top} + \ip{\cA_0, Z - zz^\top}\\
&= q_0(\hat x) + \ip{\cA_0, Z} \pm 2\ip{A_0 \hat x + b_0, y\otimes z}\\
&= 2\hat t \pm 2\ip{A_0 \hat x + b_0, y\otimes z},
\end{align*}
we deduce that $(\hat x\pm y \otimes z, 2\hat t \pm 2\ip{A_0 \hat x + b_0, y\otimes z})\in\cD_\textup{SDP}$. Then, by minimality of $\hat t$ in $\cD_\textup{SDP}$, we infer that $\ip{A_0\hat x + b_0, y\otimes z} = 0$.

We deduce that $(\hat x\pm y\otimes z, \hat t, Z - zz^\top)$ satisfies \eqref{eq:hatx_hatt_Z_system_obj}.
Furthermore, we have $\rank(Z-zz^\top)= r-1$. By induction, $(\hat x\pm y\otimes z,\hat t)\in\conv(\cD\cap H)$. 
We conclude that $(\hat x,\hat t)\in\conv(\cD\cap H)$.
\end{proof} \end{appendix}

\end{document}